\documentclass[10pt]{article}
\title{On edge-direction and compact edge-end spaces}
\author{Gustavo Boska \and Matheus Duzi \and Paulo Magalhães Júnior}

\newcommand{\Addresses}{{
  \bigskip
  \footnotesize
  G. ~Boska, \textsc{Instituto de Ci\^encias Matem\'aticas e de Computa\c c\~ao, Universidade de S\~ao Paulo\\
	Avenida Trabalhador s\~ao-carlense, 400,  S\~ao Carlos, SP, 13566-590, Brazil}\par\nopagebreak
  \textit{E-mail address}, G.~Boska: \texttt{gustavo.boska@usp.br}

  \medskip
  M.~Duzi, \textsc{Instituto de Ci\^encias Matem\'aticas e de Computa\c c\~ao, Universidade de S\~ao Paulo\\
	Avenida Trabalhador s\~ao-carlense, 400,  S\~ao Carlos, SP, 13566-590, Brazil}\par\nopagebreak
  \textit{E-mail address}, M.~Duzi: \texttt{matheus.duzi.costa@usp.br}
  
  \medskip
  P.~Magalhães Jr, \textsc{Instituto de Ci\^encias Matem\'aticas e de Computa\c c\~ao, Universidade de S\~ao Paulo\\
	Avenida Trabalhador s\~ao-carlense, 400,  S\~ao Carlos, SP, 13566-590, Brazil}\par\nopagebreak
  \textit{E-mail address}, P.~Magalhães Jr: \texttt{pjr.mat@usp.br}
 
}}

\usepackage[a4paper, total={5.5in, 10in}]{geometry}
\usepackage{amsfonts,amsmath,amscd,amssymb,stmaryrd,amsthm}
\usepackage{centernot}
\usepackage{cmap}
\usepackage[T1]{fontenc}
\usepackage[utf8]{inputenc}
\usepackage{lmodern}
\usepackage{graphicx,caption}
\usepackage{multirow}
\usepackage[english]{babel}
\usepackage{natbib}
\usepackage{xcolor}
\definecolor{color1}{gray}{1}
\usepackage{todonotes}
\usepackage[colorlinks=false]{hyperref}
\usepackage[capitalize]{cleveref}
\newcommand{\myref}[1]{\hyperref[#1]{\namecref{#1} \ref{#1}}}
\usepackage[normalem]{ulem}

\usepackage{tikz}
\usepackage{tikz-cd}
\usetikzlibrary{calc,positioning,arrows.meta,shapes.geometric,math}
\tikzset{
    buffer/.style={
        draw,
        shape border rotate=-90,
        isosceles triangle,
        isosceles triangle apex angle=60,
        node distance=2cm,
        minimum height=4em
    }
}
\tikzset{>={Latex[width=2mm,length=2mm]}}
\usepackage{ifthen}

\numberwithin{equation}{subsection}

\newtheorem{theorem}[equation]{Theorem}
\newtheorem{lemma}[equation]{Lemma}
\newtheorem{proposition}[equation]{Proposition}
\newtheorem{corollary}[equation]{Corollary}
  
\newtheorem{definition}[equation]{Definition}
\newtheorem{example}[equation]{Example}
\newtheorem*{theorem*}{Theorem}
\newtheorem*{proposition*}{Proposition}
\newtheorem*{corollary*}{Corollary}
\newtheorem*{lemma*}{Lemma}
\theoremstyle{remark}
\newtheorem{rmk}[equation]{Remark}

\DeclareMathOperator{\id}{id}

\newcommand{\st}{\, : \,}

\newcommand{\set}[1]{\left\{\, {#1} \,\right\}}
\newcommand{\seq}[1]{{\left\langle \, {#1} \,\right\rangle}}
\newcommand{\seqq}[1]{{\left\langle {#1} \right\rangle}}
\newcommand{\com}[1]{}
\newcommand{\restrict}[2]{{#1} \!\! \upharpoonright_{#2}}

\newcommand{\proofshow}{1}

\begin{document}
    
\maketitle
    \begin{abstract}
    
    Directions of graphs were originally introduced in the study of a cops-and-robbers kind of game, while the study of end spaces has been used to generalize classical graph-theoretical results to infinite graphs, such as Halin's generalization of Menger's theorem. An edge-analogue of end spaces, where finite sets of edges are used instead of vertices as separator agents to form the so-called edge-end space, has been recently used to obtain an edge-analogue of this later result. Inspired by Diestel's correspondence between directions and ends of a graph, we tackle in this paper an edge-analogue of directions, its relation with line graphs, and an edge-analogue of Diestel's correspondence. The results of this study had some implications over edge-end space compactness, which then became a target of inquiry: we thus show an edge-analogue of Diestel's combinatorial characterization for compact end spaces. Non-edge-dominating vertices play an important role in our characterization, which motivated the study of ends and directions using now finite sets of these vertices as separator agents, as done previously for edges, giving rise to other topological spaces associated with graphs. These new direction and end spaces once again motivate an analogue of Diestel's correspondence result, and further generalizations are obtained. All of these constructions define topological space-classes associated with graphs such as edge-end spaces and edge-direction spaces of graphs. The paper organizes these topological space-classes appearing throughout the text with representation results, as it was done by Kurkofka and Pitz, as well as Aurichi, Magalhães Júnior and Real. Most notably, we show that every compact edge-end space can be represented as the edge-direction space of a connected graph.

    \end{abstract}


    \section{Introduction}
    The concept of \emph{ends} appeared in a more general topological context in \cite{Freudenthal1931}. The current graphical notion of ends was presented by Halin in \cite{Halinend} and recently gained momentum, since it helps to generalize some classical combinatorial results, as in \cite{Diestel2004}, where a base of a generalized cycle space was achieved using the fundamental cycles of normal spanning trees. 

We follow \cite{diestelgraphtheory} for the usual notations and conventions for basic \emph{graph theory}. All graphs in this text shall be considered simple and undirected, so our edges are sets of two distinct vertices $e = \{v_1,v_2\}$, in which case we say $v_1$ and $v_2$ are \emph{incident} in $e$. If $G$ is a graph, a \emph{$G$-ray} is an infinite path of neighboring vertices $\seq{v_1, v_2, \dotsc}$ such that $\{v_i,v_{i+1}\} \in \mathrm{E}(G)$ for $i \in \mathbb{N}$. Given $n\in\omega$, we say that $\seq{v_{n}, v_{n+1}, \dotsc}$ is a \emph{tail} of the $G$-ray $\seq{v_1, v_2, \dotsc}$. Two rays are said to be \emph{equivalent} if for every finite $F\subset \mathrm{V}(G)$, their tails lie within the same connected component of $G\setminus F$. This relation can be characterized by the existence of a third ray intercepting both rays in infinitely many vertices. This equivalence $r \sim s $ defines classes $[r]$ called \emph{ends}. The set of $G$-ends is denoted by $\Omega(G)$. 

There is a natural topology in $\Omega(G)$: given an end $\varepsilon\in \Omega (G)$ and a finite set of vertices $F\subset V(G)$, there exists a unique connected component $C(F,\varepsilon)$ of $G\setminus F$ that contains the tails of all representatives $r \in \varepsilon$. In this case, if we define  \[\Omega(F,\varepsilon)=\lbrace \eta \in \Omega (G) \,:\, C(F, \eta) = C(F, \varepsilon) \rbrace\subset \Omega(G) ,\]
then the collection $\lbrace \Omega (F, \varepsilon) : F\in [\mathrm{V}(G)]^{<\aleph_0}, \varepsilon\in \Omega (G)\rbrace$, where $[\mathrm{V}(G)]^{<\aleph_0}$ denotes the family of finite subsets of $\mathrm{V}(G)$, can be construed as a basis of the (now topological) \emph{end space}. We say a vertex $v$ \emph{dominates a ray} $r$ if the tail of $r$ lies withing the same connected component of $G\setminus F$ for every $F\subset \mathrm{V}(G)\setminus \{v\}$ (or, equivalently, if there are infinitely many vertex-disjoint paths from $v$ to $r$), and it \emph{dominates an end} when it dominates one of its representatives. We denote $\mathrm{S}_{G,F}$ as the set (or discrete space) whose elements are sets of vertices of the infinite connected components of $G \backslash F$. 

Ends and end spaces are well studied, for example, in \cite{Endspacesandspanningtrees}, \cite{representation} and \cite{pitz}. If we make an analogue of end spaces where edge-connectivity becomes the protagonist, one quickly arrives at edge-end spaces. This construction appears in \cite{lav} and has been sparsely explored in the literature. In the following, we define the edge-end space. Two $G$-rays are said to be \emph{edge-equivalent} if, for every finite $F\subset \mathrm{E}(G)$, their tails lie within the same connected component of $G\setminus F$. This equivalence relation's classes are the \emph{edge-ends}. The set of edge-ends of a graph $G$ is denoted by $\Omega_E(G)$. By a process similar to the one made in the end space, we define a topology in $\Omega_E(G)$: given an edge-end $\varepsilon\in \Omega_E (G)$ and a finite set of edges $F\subset \mathrm{E}(G)$, there is a unique connected component $C(F,\varepsilon)$ of $G\setminus F$ that contains a tail of $\varepsilon$. Hence, if we define $$\Omega_E(F,\varepsilon)=\lbrace \eta \in \Omega_E (G) : C(F, \eta) = C(F, \varepsilon) \rbrace ,$$ then the family $\lbrace \Omega_E (F, \varepsilon) : F\in [E(G)]^{<\aleph_0}, \varepsilon\in \Omega_E (G)\rbrace$ can be considered as a basis for the \emph{edge-end space} $\Omega_E (G)$. 

We say a vertex $v$ \emph{edge-dominates a ray} $r$ if no finite set of edges separate them, and it \emph{edge-dominates an end} when it dominates one of its representatives. We say that a vertex is \emph{timid} if it does not edge-dominate any ray in $G$. We denote by $\mathrm{t}(G)$ the set of timid vertices of $G$. Although for locally finite graphs $\Omega(G)\approx \Omega_E(G)$, \myref{nonadmissiblevertexset} shows us this is not the case in general.

\begin{figure}[ht]
\centering
\begin{tikzpicture}
    \def\radius{0.15}
    \def\dy{0.7}
    \def\dx{1}

    \begin{scope}[shift = {(11*\dx,0)}]
        \draw[->] (-0.5*\dx,\dy)-- (7.5*\dx,\dy);
        \draw[->] (-0.5*\dx,-\dy)-- (7.5*\dx,-\dy);

        \foreach \i in {0,...,7}
        {
            \draw (3.5*\dx,0) -- ({\i*\dx},\dy);
            \draw (3.5*\dx,0) -- ({\i*\dx},-\dy);
        }
    \end{scope}
\end{tikzpicture}
\caption{For this graph $G$, $\Omega(G)$ has two distinct ends, but $\Omega_E(G)$ has only one edge-end.}
\label{nonadmissiblevertexset}
\end{figure}
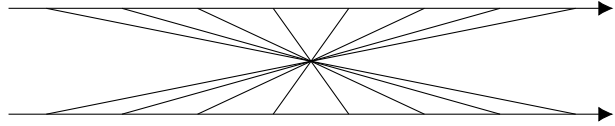

It was recently shown that every edge-end space is an end space, but not the other way around:

\begin{theorem*}[Theorem 1.1 in \cite{aurichi2024topologicalremarksendedgeend}]
    \sloppy The class $\Omega_E=\lbrace \Omega_E(G) : G \text{ is a graph}\rbrace$ is a proper subclass of $\Omega= \lbrace \Omega (G): G \text{ is a graph}\rbrace$, i.e., every edge-end space is the end space of some graph, but the converse does not hold.
\end{theorem*}

The fact above suggests that edge-end spaces comprise a more restricted class of topological spaces then that of end spaces. This class of spaces also gives a Menger-type result, proved in \cite{real}. These results motivate further inquiry over properties of the edge-end spaces. 

A \emph{direction of a graph $G$} is a map $\rho$ that takes a finite set of vertices $F$ and gives us a non-empty connected component $\rho(F)$ of $G \backslash F$ such that if $A \subset B$ are finite sets of vertices then $\rho(B) \subset \rho(A)$ (in particular, note that $\rho(F)$ must always be infinite). This notion came to light, for example, in the study of strategies in a game of criminals evading the police in \cite{Robertson1991_CopsAndRobbers}. As it was done for ends, using finite edge-sets as the means for separation leads to the definition of the edge-directions $\mathcal{D}_E(G)$. The set $\mathcal{D}_E(G)$ becomes a topological space, since it appears as the inverse limit of the system of discrete spaces of connected components. A correspondence between ends and directions of $G$ was established in \cite{DIESTEL2003197}, so we begin by posing the question if this fact has an edge-analogue, which shall be given in \myref{edgedirectionsareendsofedgegraph}. 

Let $G$ be any graph, define the \emph{line-graph of $G$}, denoted by $G'$, to be a graph whose vertices are in correspondence with $G$-edges (that is, $\mathrm{V}(G')\doteq \{v_e\}_{e \in \mathrm{V}(G)}$) such that $\{v_{e_1},v_{e_2}\}$ is an edge of $G'$ exactly when $e_1$ shares a vertex with $e_2$. Using the notation $X \approx Y$ for the homeomorphism relation between spaces, we prove in \myref{edgedirection}: 

\begin{theorem*}[\ref{edgedirectionsareendsofedgegraph}]
    \sloppy For any graph $G$, $\mathcal{D}_E(G) \approx \Omega(G')$. Therefore
$\mathcal{D}_E=\lbrace \mathcal{D}_E(G): G \text{ is a graph}\rbrace = \Omega'= \lbrace \Omega(G'): G \text{ is a graph}\rbrace$. 
\end{theorem*}

The above result allows us to see directions of a graph as ends of another one, making some calculations simpler. The following are some consequences, the second of which being an edge-analogue of \cite{DIESTEL2003197}'s Theorem 2.2:

\begin{corollary*}[\ref{COR_EdgeDirCompact}]
    For every connected graph $G$, the space $\mathcal{D}_E(G)$ is compact, meaning that the end space of every connected line-graph is compact.
\end{corollary*}

\begin{corollary*}[\ref{directionrepresentation}]
    Let $\rho \in \mathcal{D}_E(G)$ be any edge-direction, then the following is complementary:
    \begin{itemize}
        \item[(i)] There is a ray $r$ such that, for any finite edge-set $F$, $\rho(F)$ is the component of $G \backslash F$ that has an $r$-tail.
        \item[(ii)] There is a timid vertex of infinite degree $v$ such that, for any finite edge-set $F$, $\rho(F)$ is the component $G \backslash F$ containing $v$.
    \end{itemize}
\end{corollary*}

Note that case (ii) illustrates a different behavior of edge-directions in contrast to vertex directions: they may not always be represented by rays. A \emph{rayless direction} is thus a direction that can only be represented by a timid vertex with infinite degree. We thus obtained:

\begin{corollary*}[{\ref{raylesscharact}}]
    The following are equivalent, for a connected graph $G$:
    \begin{itemize}
        \item[(i)] The edge-end space of a graph $G$ is compact. 
        \item[(ii)] The canonical $\iota_G:\Omega_E(G) \rightarrow \mathcal{D}_E(G)$ includes $\Omega_E(G)$ as a closed subset of $\mathcal{D}_E(G)$.
        \item[(iii)] The set of rayless directions is an open set of $\mathcal{D}_E(G)$.
        \item[(iv)] Every timid vertex of infinite degree $v$ admits a finite edge-set $F$ such that $v \in C \in \mathrm{S}_{G,F}$ and $C$ is rayless.
    \end{itemize}
\end{corollary*}

In \myref{directionedgerep} we ask if edge-direction spaces can be combinatorially realized as edge-end spaces of some graph, and the answer is again positive. Given a graph $G$, for each edge-direction (end of the line graph of $G$) that is \emph{not} represented by an edge-end in $G$ we add a ray connecting a star representing this direction, obtaining the \emph{completion graph $\tilde{G}$} of $G$. We thus obtain

\begin{theorem*}[{\ref{completiontheorem}}]
    Let $G$ be a graph and $\tilde{G}$ its completion. Then $\Omega(G')\approx \mathcal{D}_E(G) \approx \Omega_E(\tilde{G})$. 
\end{theorem*}

The graph $\tilde{G}$ above is \emph{direction-complete}, meaning that all of its edge-directions are represented by rays. This proves that the end space of line-graphs are a subclass of the \emph{class of edge-end spaces} $\Omega_E$. \myref{notendspaceofedgegraph} shows that this subclass relation is strict:

\begin{theorem*}[\ref{notendspaceofedgegraph}]
   \sloppy The inclusion of classes $\mathcal{D}_E=\lbrace \mathcal{D}_E(G): G \text{ is a graph}\rbrace \subsetneq \Omega_E= \lbrace \Omega_E(G): G \text{ is a graph}\rbrace$ is a strict one.
\end{theorem*}

Hence, a simpler analogue of \myref{THM_HGHomeo} from \cite{aurichi2024topologicalremarksendedgeend} which considers the line-graph construction in place of the more complicated construction which was presented cannot easily be done.

Motivated by \myref{raylesscharact}, we also search for an edge-analogue of the following result from \cite{Endspacesandspanningtrees} that characterizes graphs whose end spaces are compact:

\begin{theorem*}[Corollary 4.4 in \cite{Endspacesandspanningtrees}]\label{THM_CompactCharacterization}
    The end space of a graph $G$ is compact if, and only if, for every finite $F\subset \mathrm{V}(G)$, the collection $\tilde{\mathrm{S}}_{G,F}$ of non-rayless connected components is finite.
\end{theorem*}

Applying \myref{THM_CompactCharacterization} to the graph $H_G$ constructed in \myref{THM_HGHomeo} (\cite{aurichi2024topologicalremarksendedgeend}), taking its construction into account, we arrive at the following edge-analogue:

\begin{theorem*}[\ref{compactnesscharacterization}]
    Let $G$ be any graph, the following are equivalent:
    \begin{itemize}
        \item[(i)] The space $\Omega_E(G)$ is compact.
        \item[(ii)] For every finite set of timid vertices $F$, the collection of connected components of $G \backslash F$ with a ray is finite (\myref{compactnesscharacterization}).
    \end{itemize}
\end{theorem*}

Following our theme of searching for (strict) subclass relationships between these topological space classes associated with graphs, in \myref{compactnesscharacterization} we ask ourselves if the class of compact edge-end spaces is equal to the class of edge-direction spaces of connected graphs. \myref{THM_CompectEdgeEndDirection} answers this positively:

\begin{theorem*}[\ref{THM_CompectEdgeEndDirection}]
    The edge-end space of a graph $G$ is compact if, and only if, there is some connected graph $H$ such that $\Omega_E(G)\approx\Omega_E(H)\approx \mathcal{D}_E(H)$.
\end{theorem*}

\newcommand{\myrectangle}[5]{
    \draw (#1,#2) rectangle ++ (#3,#4);
    \node[shift = {({-#3*0.5},{-0.5*#4})}] at (#1+#3,#2+#4) {\large\textbf{#5}};
}
\newcommand{\myrectangletwo}[6]{
    \draw (#1,#2) rectangle ++ (#3,#4);
    \node[shift = {({-#3*0.5},{-0.3*\dy})}] at (#1+#3,#2+#4) {\large\textbf{#5}};
    \node[shift = {({-#3*0.5},{-0.75*\dy})}] at (#1+#3,#2+#4) {\footnotesize #6};
}

\begin{figure}[ht]
\centering
\begin{tikzpicture}
    \def \dx {2.8}
    \def \dy {1.5}

    \myrectangletwo{0}{0}{5*\dx}{5*\dy}{$\mathbf{\Omega\overset{\text{Theorem 2.2 in \cite{DIESTEL2003197}}}{=}\mathcal{D}}$}{Alexandroff duplicate of the cantor space \cite{aurichi2024topologicalremarksendedgeend}}
    
    \myrectangletwo{0.1*\dx}{0.1*\dy}{4.8*\dx}{3.75*\dy}{$\mathbf{\Omega_E\overset{\myref{THM_EdgeEndsHomeoTimidEnds}}{=}\Omega_{\mathrm{t}}\overset{\myref{COR_TimidDirectionsHomeoTimidEnds}}{=}\mathcal{D}_{\mathrm{t}}}$}{\myref{notendspaceofedgegraph}}
    
    \myrectangletwo{0.2*\dx}{0.2*\dy}{4.6*\dx}{2.50*\dy}{$\mathbf{\mathcal{D}_E \overset{\myref{edgedirectionsareendsofedgegraph}}{=} \Omega'}$}{\myref{rmk1}}

    \myrectangle{0.3*\dx}{0.3*\dy}{4.4*\dx}{1.25*\dy}{$\mathbf{\mathcal{D}_E \text{ of connected graphs } \overset{\text{\myref{THM_CompectEdgeEndDirection}}}{=} \text{ compact } \Omega_E \text{ spaces}}$}

\end{tikzpicture}
\caption{A diagram depicting the relationship between the space classes.}
\label{vennfig}
\end{figure}
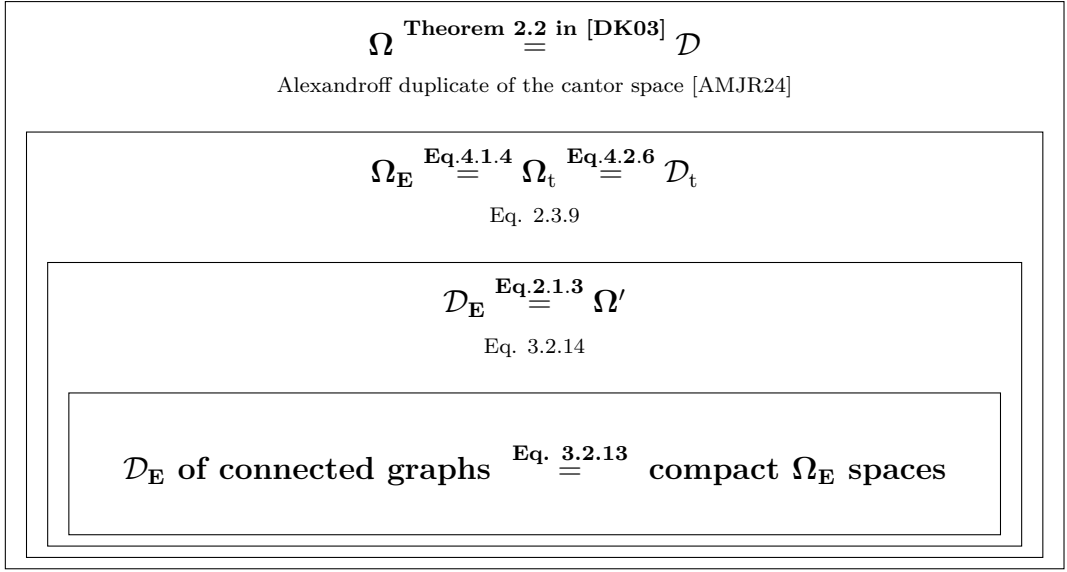

Timid vertices played a key role in \myref{compactnesscharacterization}. Inspired by this, timid vertices and timid connectivity become the center of our attention in \myref{SecTimid-end1}. Let $\Omega_{\mathrm{t}}(G)$ denote the quotient of the set of $G$-rays by the $\mathrm t$-equivalence in which two rays are identified if no finite set of timid vertices can separate tails of said two rays. The elements of $\Omega_{\mathrm{t}}(G)$ shall be called \emph{timid-ends} and the basic open sets $\Omega_{\mathrm t}(F,\varepsilon)$ in $\Omega_{\mathrm{t}}(G)$ are defined just as in $\Omega(G)$, but using finite subsets $F\subset \mathrm{t}(G)$ and timid-ends. 

In locally finite graphs, the end, edge-end and timid-end spaces are all homeomorphic, but this may not be the case for every graph, as illustrated in \myref{fig:enter-label}. However, we show that the edge-end space class coincides with the timid-end space class: 

\begin{theorem*}[\ref{THM_EdgeEndsHomeoTimidEnds}]
    The class $\Omega_{\mathrm{t}}$ of timid-end spaces is the same as the class $\Omega_E$ of edge-end spaces. 
\end{theorem*}

\begin{figure}
    \centering
    \begin{tikzpicture}
        \def \a {1}
        \draw (0,0) circle (\a);
        \draw (2.5*\a,0) circle (\a);
        \draw (4.5*\a,0) circle (\a);

        \draw (\a,0) -- (1.5*\a,0);

        \draw[fill = white] (\a,0) circle (0.1*\a);
        \draw[fill = white] (1.5*\a,0) circle (0.1*\a);
        \draw[fill = white] (3.5*\a,0) circle (0.1*\a);

        \node at (0,0) {$K_\omega$};
        \node at (2.5*\a,0) {$K_\omega$};
        \node at (4.5*\a,0) {$K_\omega$};
    \end{tikzpicture}
    \caption{For this graph $G$, one has that $\Omega_\mathrm{t}(G)$ has one point, $\Omega_E(G)$ has two and $\Omega(G)$ has three.}
    \label{fig:enter-label}
\end{figure}
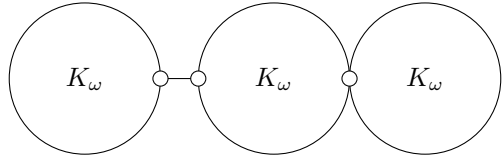

We make an edge-analogue study of \cite{DIESTEL2003197}, establishing relationships between ends and directions in \myref{directionedgerep}, for edge-ends vs edge-directions, and in \myref{timid-direciton-end2}, for timid-ends and timid-directions. This last result raises the natural question of its validity when we exchange the timid vertex set for another one. This motivates us to define the \emph{$U$-end space} ($\Omega_U(G)$) and the \emph{$U$-direction space} ($\mathcal{D}_U(G)$) for a graph $G$ and $U\subset \mathrm{V}(G)$. For a $G$-ray $r$, we denote by $\rho_r$ the $U$-direction which, for every $F\subset U$, chooses the connected component of $G\setminus F$ containing a tail of $r$. In this case, we show that

\begin{theorem*}[\ref{THM_piDSurj}]
    Let $U\subset \mathrm{V}(G)$ be given. Then for every $\rho\in \mathcal{D}_U(G)$ there exists a $G$-ray $r$ such that $\rho_r = \rho$ if, and only if, the set of vertices in $G$ which can be separated from any $G$-ray by finite subsets of $U$, but cannot be isolated into a finite component of $G$ by any finite $F\subset U$,  is contained in $U$.
\end{theorem*}

\begin{corollary*}[\ref{COR_TimidDirectionsHomeoTimidEnds}]
    For every graph $G$, $\Omega_{\mathrm{t}}(G)\approx  \mathcal{D}_{\mathrm{t}}(G)$.
\end{corollary*}

We finish this section with some standard graphical notions which will be of use throughout this paper. A \emph{comb} (of $G$) is a pair $(r,\{d_i\}_{i \in \omega})$ and a collection of pairwise disjoint paths $\{P_i\}_{i \in \omega}$ where $r$ is a $G$-ray, the comb's \emph{spine}, and $d_i$ the combs \emph{teeth}. An star centered in $c$ and with tips $T$ is a collection of paths from $c$ to $t \in T$ such that for any pair of paths $P_1\cap P_2 = \{c\}$. The star is \emph{infinite} if $T$ is infinite. In this case, the following can be found in \cite{diestelgraphtheory}:

\begin{lemma}[Star-comb lemma]
\label{starcomb}
    Let $ G $ be connected and an infinite $ D \subset \mathrm{V}(G) $. If there is no comb with infinitely many teeth in $D$ then there is a star with infinitely many tips in $ D $.
\end{lemma}

    \section{Representations of edge-direction spaces}
    
\paragraph{}
In \cite{DIESTEL2003197} a correspondence $\varepsilon \in \Omega(G) \leftrightarrow \rho_\varepsilon \in \mathcal{D}(G)$ is constructed for any graph $G$. In the following, we attempt to make a similar study about the edge-end case. Let $G$ be any graph, consider the system \[ F_G \doteq \left( \{\mathrm{S}_{G,F}\}_{F \in [\mathrm{V}(G)]^{<\aleph_0}},\{\phi_{F_1,F_2}\}_{F_1 \subset F_2 \in [\mathrm{V}(G)]^{<\aleph_0}}, [\mathrm{V}(G)]^{<\aleph_0}\right)\] given by the following associations. 

\begin{itemize}
    \item We define $\mathrm{S}_{G,F}$ to be the discrete space whose points are sets of vertices of the connected components of $G \backslash F$. 
    \item Given $F_1 \subset F_2$ we define a map $ \phi_{F_1,F_2}^G: \mathrm{S}_{G,F_2} \to \mathrm{S}_{G,F_1} $ that maps a component $C \in \mathrm{S}_{G,F_2}$ into the only component $\phi_{F_1,F_2}^G(C) \in S_{G,F_1}$ that contains it.
\end{itemize}  

It is straight forward to check that elements of $\varprojlim \mathrm{S}_{G,\bullet}$ are in bijection with graph-theoretic \emph{directions} of $G$. One can check that basic open sets in the usual topology $\Omega(G)$ are in correspondence with basic open sets of $\varprojlim F_G$, giving us a description of $\Omega(G)$ as an inverse limit. There is no loss of generalization to use \[\tilde{\mathrm{S}}_{G,F} \doteq \{C \in \mathrm{S}_{G,F} \st C \text{ has at least a ray}\}.\] 

If $C$ is rayless, the neighborhood defined by $\pi_{F}^{-1}(C)$ is empty. Furthermore, if there is an $F$ with infinite $\tilde{S}_{G,F}$ then the covering $\{\pi_{F}^{-1}(C)\}_{C \in \tilde{S}_{G,F}}$ has no finite sub-cover. It is a known fact that the inverse limit of Hausdorff compact spaces is compact (\cite{eng}).

\subsection{Edge-direction spaces as end spaces}
\label{edgedirection}

\newcommand{\im}{\mathrm{im}}

We begin by defining the following direct analogue of graph theoretical directions for edge removal. We also recall the definition of edge-domination, and fix some notions we shall use from here on. Recall (following \cite{eng}) that an \emph{inverse system} is a triple \[F \doteq (\{X_p\}_{p \in \mathbb{P}},\{\phi_{p_1p_2}:X_{p_2}\to X_{p_1}\}_{p_1 \leq p_2}, \mathbb{P})\] where $\mathbb{P}$ is an order, $X_p$ is a topological space for each $p \in \mathbb{P}$ and $\phi_{p_1p_2}$ is a continuous map, called \emph{transition} or \emph{attaching} maps. We ask $\phi_{pp}= \id_{X_p}$, that $\phi_{p_1p_2} \circ \phi_{p_2p_3} = \phi_{p_1p_3}$ and that $\mathbb{P}$ is \emph{upward directed order}, meaning that for any pair $p_1,p_2 \in \mathbb{P}$ there is $p_3 \in \mathbb{P}$ such that $p_3 \geq p_1,p_2$. Define the \emph{limit} of this system as
\[ \varprojlim F = 
    \left\{
    (x_p)_{p \in \mathbb{P}} \in \prod_{p \in \mathbb{P}}X_p \st 
    \forall p_1 \leq p_2 
        (\phi_{p_1p_2}(x_{p_2}) = x_{p_1}) 
    \right\}
\]
taken with the subspace topology from the product. Note that the canonical projections $\pi_p:\varprojlim F \to F(p)$ provide us a natural sub-basis of the inverse limit, given as $\{\pi_p^{-1}(A)\}_{p \in \mathbb{P},\, A \in \tau_{F(p)}}$. 

\begin{definition}[Edge directions]
    Let $G$ be a graph, define $\mathrm{S}^E_{G,F}$ to be the discrete space of the infinite connected components of $G\backslash F$ for $F \in [\mathrm{E}(G)]^{<\aleph_0}$ and transition maps $\varepsilon^G_{F_1,F_2}$ defined in the same way as for the vertex construction. The inverse limit of the system $(\mathrm{S}^E_{G,F})$ will give us the topological space of \emph{edge-directions} of $G$, denoted by $\mathcal{D}_E(G)$. This topological space is an \emph{edge-direction space}.
\end{definition}

It is easy to see that $\set{\mathcal{D}_E(F,\rho):F\in [\mathrm{E}(G)]^{\aleph_0}, \, \rho\in \mathcal{D}_E(G)}$ is a basis for $\mathcal{D}_E(G)$, with
\[\mathcal{D}_E(F,\rho) \doteq \set{\rho'\in \mathcal{D}_E(G)\st \rho'(F) = \rho(F)}.\]

We note that, although it need not to be specified in the definition of $\mathcal{D}(G)$ that the maps choose from \emph{infinite} components (for this follows as an immediate consequence), this is not the case for $\mathcal{D}_E(G)$. Thus, we have chosen to make such specification to maintain the intuition of directions being \emph{points at infinity} for edge-directions.

As in the vertex-centered case (\cite{DIESTEL2003197}) each edge-end $\varepsilon = [r] \in \Omega_E(G)$ \emph{represent a direction} $ \rho_\varepsilon \in \mathcal{D}_E(G)$, where $\rho_\varepsilon(F) \in  \mathrm{S}^E_{G,F}$ is the component with an $r$-tail, since two rays $r_1$ and $r_2$ define the same direction $\rho_{[r_1]} = \rho_{[r_2]}$ if and only if they are edge-equivalent. We do \emph{not} have, as it is the case for the vertex construction, that $\Omega_E(G) \approx \mathcal{D}_E(G)$ in general. A simple infinite star is the counter-example: this is a rayless graph in which an edge-direction exists. This different behavior motivates the following definition.

\begin{definition}\label{DEF_CanonicalEmbed}
    Let $G$ be any graph. The \emph{canonical embedding} $\iota_G:\Omega_E(G)\hookrightarrow \mathcal{D}_E(G)$, $\iota(\varepsilon) \doteq \rho_\varepsilon$ is well defined. In order to see that it is continuous, note that if $\rho_\varepsilon \in \pi_F^{-1}(C)$ is any basic open set containing an edge-direction in the image, then $\iota_G(\Omega(\varepsilon,F)) \subset \pi_F^{-1}(F)$. A direction $\rho \in \mathcal{D}_E(G)\backslash \iota_G(\Omega_E(G))$ is called a \emph{rayless direction}. 
    
    A vertex $v$ with infinite degree also \emph{represents a direction} defined by $\rho_v(F) \doteq C \in \mathrm{S}^E_{G,F}$ being the component with $v \in C$. Notice $\rho_v = \rho_\varepsilon$ exactly when $v$ \emph{edge-dominates} a ray that represents $\varepsilon$. We say that a vertex is \emph{timid} if it does not edge-dominate any ray in $G$. Let $\mathrm{t}(G)$ be the collection of timid vertices from $G$. 
\end{definition}

The following gives a representation for edge-direction spaces as end spaces. 

\begin{theorem}[Edge-directions are ends of the line-graph]
\label{edgedirectionsareendsofedgegraph}
    For any graph $G$ the homeomorphism relation $\mathcal{D}_E(G) \approx \Omega(G')$ holds.
\end{theorem}

\if \proofshow 1
\begin{proof} 
    Notice the following correspondence holds: \[F \in [\mathrm{E}(G)]^{<\aleph_0} \leftrightarrow F'\doteq \{v_e \st e \in F\} \in [\mathrm{V}(G')]^{<\aleph_0}.\] 
    
    We claim that the systems $\mathrm{S}^E_{G,F} $ and $ \mathrm{S}_{G',\{v_e\}_{e \in F}}$ are equivalent. \\
    
    Indeed, let $C$ be a connected component of $G \backslash F$ where $F$ is a finite set of edges. Take two edges $e_1,e_2$ in $C$, between the vertices $v_1 \in e_1$ and $v_2 \in e_2$. There shall be a $G$-path $P$ from $v_1$ to $v_2$ if and only if there is a $G'$-path between $v_{e_1}$ and $v_{e_2}$. This gives rise to the correspondence \[ C \in \mathrm{S}_{G,F}^{E} \leftrightarrow C' \doteq \{v_e \st e \in \mathrm{E}(G[C])\}\in \mathrm{S}_{G',F'}.\] 
    
    Furthermore the inclusion of finite sets of edges $F_1 \subset F_2 $ happens if and only if $F_1' \subset F_2'$ and the transition maps satisfy \[\varepsilon^G_{F_1,F_2}(C_2) = C_1 \iff \phi^{G'}_{F_1',F_2'}(C_2') = C_1'.\] 
    
    This proves the correspondence $\bullet': C \in \mathrm{S}_{G,F}^{E} \leftrightarrow C' \in \mathrm{S}_{G',F'}$ makes the following diagram commute:
\[
\begin{CD}
\mathrm{S}_{G,F_2}^E  		@>\bullet'>> 	\mathrm{S}_{G',F_2'}   		\\
@V \varepsilon^G_{F_1,F_2} VV 			@VV \phi^{G'}_{F_1,F_2} V  	\\
\mathrm{S}^E_{G,F_1}  		@>\bullet'>> \mathrm{S}_{G',F_1'}  \\
\end{CD}
\]
and these systems should, by 2.5.10 of \cite{eng}, produce the same inverse limit $\mathcal{D}_E(G) \approx \Omega(G')$.
\end{proof}
\fi

This means that the class of edge-direction spaces of graphs is the same as the class of end spaces of line-graphs. This also means the canonical inclusion can be seen as $\iota_G:\Omega_E(G) \hookrightarrow \Omega(G')$. We shall use both disguises $\rho \in \mathcal{D}_E(G) \leftrightarrow \varepsilon \in \Omega(G')$ in our future constructions, depending on what is more convenient. The following example shows it is not always metrizable:

\begin{example}[A non-metrizable edge-direction space]
    \label{nonmetdirection}
    Let $G$ have vertex set $\aleph_1 \times \omega$. Add just enough edges to make $\aleph_1 \times \{0\}$ into a complete sub-graph and $\langle (\gamma,0),(\gamma,1),(\gamma,2),\dots\rangle$ into a ray for each $\gamma \in \aleph_1$. Denote by $e_{\gamma,n} \doteq \{(\gamma,n),(\gamma,n+1)\}$ one of these edges. The end space of this graph is the one point compactification of the discrete space with $\aleph_1$ points $\aleph_1 \cup \{\infty\}$. The $\infty$ point in this description is associated with the end $\aleph_1 \times \{0\}$. This is a non-metrizable space because $\infty$ does not assume a countable local base. Now consider the following partition of $\mathrm{E}(G)$: take the set $K \subset \mathrm{E}(G)$ consisting of all the edges between vertices of $\aleph_1 \times \{0\}$, and singletons $\{e_{\gamma,n}\}$ for each edge of the $\gamma$-ray. This partitions $\mathrm{E}(G)$ into complete sub-graphs in a way that every vertex $v$ is in at most two elements of this partition. By Theorem 8.2 of \cite{harary}, this is a line-graph.
\end{example}

\subsection{Direction representation}
\label{directionrepsection}

\paragraph{}
Although useful in the topological context, directions are difficult to manipulate combinatorially because they lack the 'ray representative' that is frequently used in the end-theoretic sense. The line-graph end characterization should be useful in these situations. Within this brief section, we use the line-graph end representation of the edge-directions to describe them in an edge-analogue of \cite{DIESTEL2003197}'s bijection result.

\begin{definition}
    Let $G$ be a graph and $U \subset \mathrm{V}(G')$ be a collection of vertices. Let $E = \{e \in \mathrm{E}(G) \st v_e \in U\}$ and define $U^* \doteq G[E]$. 
\end{definition}

If $e' \doteq \{v_{e}, v_f\} \in \mathrm{E}(G')$ then the (possibly) three end-points of $\{v_e,v_f\}^*$ is connected. It is not difficult to use this to prove that if $G'[V]$ is connected, then $V^*$ is connected. In particular, if $r$ is a $G'$ ray, $r^* \leq G$ is connected. In what follows, we fix a canonical way of mapping a $G$-ray into a $G'$ ray. Let $r = \langle v_1,v_2,v_3,\dots \rangle $ be a $G$-ray, observe $\{v_i,v_{i+1}\} \doteq e_i$ is an edge so $\overline{r} \doteq \langle v_{e_1},v_{e_2},\dots\rangle$ is the $G'$-ray it gives rise to. A similar construction can be done with infinitely many edges incident on a single vertex $v \in G$.


\begin{corollary}[{of \myref{edgedirectionsareendsofedgegraph}}]
\label{directionrepresentation}
    Let $\rho \in \mathcal{D}_E(G)$ be any direction, then the following is complementary:
    \begin{itemize}
        \item[(i)] There is a ray $r$ such that, for any finite edge-set $F$, $\rho(F)$ is the component of $G \backslash F$ that has an $r$-tail.
        \item[(ii)] There is a timid vertex of infinite degree $v$ such that, for any finite edge-set $F$, $\rho(F)$ is the component $G \backslash F$ containing $v$.
    \end{itemize}
\end{corollary}

\if \proofshow 1


\begin{proof}
    From \myref{edgedirectionsareendsofedgegraph} every direction $\rho \in \mathcal{D}_E(G)$ corresponds to an end $\varepsilon = [r'] \in \Omega(G')$. 
    \begin{itemize}
        \item[(i)] If $(r')^\ast$ has a $G$-ray $r = \langle v_1,v_2,v_3,\dots \rangle $, then $r$ induces a $G'$-ray $\overline{r}$. It is clear that $\overline{r} \cap r'$ have infinitely many edges, so the edge-equivalence $\overline{r} \sim_E r'$ holds, meaning $[\overline{r}] = \varepsilon$ and so $\rho_{[r]} = \rho$. 
        \item[(ii)] Otherwise, by \myref{starcomb} (star-comb) we find a $v$-centered star contained in $(r')^*$. The infinitely many edges incident to $v$, call them $r$, induce a $G'$-ray $\overline{r}$ that by the same argument is edge-equivalent with $r'$, and therefore $[\overline{r}] = \varepsilon$. Hence, again via our homeomorphism, we get $ \rho_v = \rho $. 
    \end{itemize}
    
    Observe that if $\rho_\varepsilon = \rho = \rho_v $ then $v$ edge-dominates $\varepsilon$, so $\rho$ is a rayless direction if, and only if, $v$ is timid.
\end{proof}
\fi

In particular, the above establishes a bijection between timid vertices of infinite degree and rayless edge-directions.

\subsection{The completion graph}
\label{directionedgerep}

In this section, we ask if every edge-direction space is an edge-end space and if it is possible to construct some sort of completion $\tilde{G}$ of $G$ such that $G \leq \tilde{G}$ and $G$ is direction-complete. We shall answer both inquiries positively by finding a graph $\tilde{G}$ with $\mathcal{D}_E(\tilde{G}) \approx \Omega_E(\tilde{G}) \approx \mathcal{D}_E(G)$. This shall be accomplished by adding ends to $G$ to fill the gap $\mathcal{D}_E(G)\backslash \iota(\Omega_E(G))$ of rayless directions. 

\paragraph{Completion graph construction} Let $G$ be any graph, we construct a graph $\tilde{G} \geq G$ by adding edges to $G$. For each rayless direction $\varepsilon = [r] \in \Omega(G')$, consider the following claims and constructions. 

\begin{itemize}
    \item[(i)] By \myref{directionrepresentation} we find a star centered at $v_\varepsilon$ and contained in $r^*$ (an edge-direction given by a vertex of infinite degree $\rho_v$ that corresponds to $[r]$ via \myref{edgedirectionsareendsofedgegraph}).
    \item[(iii)] We construct the graph $\tilde{G}$, with the same vertex set by adding \emph{new edges} $\tilde{e}$ in the following way: add just enough edges to $G$ to make a \emph{new ray} $\tilde{r}_\varepsilon$ out of any (countably) infinite set of neighbors of $v_\varepsilon$ for each rayless direction $\varepsilon \in \Omega(G')$. 
    \item[(iv)] At the end of this process we obtain $\tilde{G}$, a\footnote{A graph may have multiple completions : the choice $\varepsilon \mapsto v_\varepsilon$, the choice of star centered on this vertex, the choice of the set of vertices that will compose a ray and the order they will appear in the ray are all arbitrary.} \emph{completion of the graph $G$}. We shall use the convention to denote \emph{new edges} with a tilde $\tilde{e} \in \tilde{\mathrm{E}}(G) \doteq \mathrm{E}(\tilde{G})\backslash \mathrm{E}(G)$ and represent them with dashed lines in our illustrations. The \emph{old edges} $e \in \mathrm{E}(G)$ are naturally still edges from $\tilde{G}$. 

    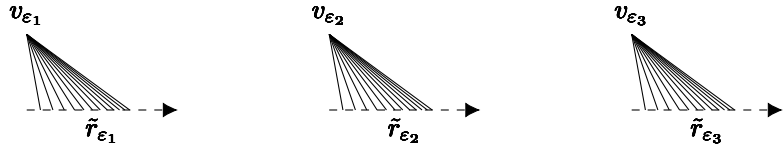
\begin{figure}[h!]
    \centering
    \begin{tikzpicture}
        \def\radius{0.15}
        \def\dy{1}
        \def\dx{2}
    
        \draw[dashed,->] (\dx,0)-- (2*\dx,0);
        \draw[dashed,->] (3*\dx,0)-- (4*\dx,0);
        \draw[dashed,->] (5*\dx,0)-- (6*\dx,0);
    
        \foreach \i in {1,3,5}{
            \foreach \n in {1,...,12}{
                \pgfmathtruncatemacro{\j}{(\i+1)/2};
                \node[above] at (\i*\dx,\dy) {$v_{\varepsilon_\j}$};
                \node[below] at ({\i*\dx + \dx/2},0) {$\tilde{r}_{\varepsilon_\j}$};
                \draw (\i*\dx,\dy) -- ({(\i+1-(1.1)^(-\n))*\dx},0);
            }
        }
    \end{tikzpicture}
\caption{New rays for rayless directions $\varepsilon_1,\varepsilon_2$ and $\varepsilon_3$ of $G'$.}
\end{figure}
    
    \item[(v)] It cannot be the case that two distinct $\varepsilon_1, \varepsilon_2$ have the same chosen vertex $v_{\varepsilon_1} = v_{\varepsilon_2}\doteq v$, because there would be $r_i \in \varepsilon_i$ such that there is a $v$-centered star included in both $r_i^\ast$, but then every vertex of $r_1$ is adjacent to every vertex of $r_2$ and that is testified by the fact that $v$ is a common vertex between every edge, meaning $\varepsilon_1 = \varepsilon_2$.
\end{itemize}

Note that \myref{directionrepresentation} gives rise to a map $\psi_G:\mathcal{D}_E(G) \to \Omega_E(\tilde{G})$, since now we do not have any rayless directions. More precisely, $\psi_G(\rho) \doteq [r]$ if this direction comes from a $G$-ray, and we take the edge-end given by the new ray $\psi_G(\rho) \doteq [\tilde{r}_\varepsilon]$ for a rayless direction. In what follows, we will argue this is a homeomorphism.

\begin{proposition}
\label{directioncompleteness}
    The map $\psi_G$ is surjective. In other words, every edge-end of $\Omega_E(\tilde{G})$ is edge-equivalent to either a $G$-ray or to a new ray.
\end{proposition}

\if \proofshow 1
\begin{proof}
    Let $\tilde{r} = \langle e_0,e_1,\dots \rangle$ be a $\tilde{G}$-ray represented here as a sequence of edges instead of as a sequence of vertices. 
    \begin{itemize}
        \item If eventually the edges of $\tilde{r}$ are all old, we are done by taking a tail.
        \item Suppose now there are infinitely many new edges. Enumerate pairs $\ell_n \doteq (P_n \doteq \langle e_{k_n},e_{k_n+1},\dots,e_{p_n}\rangle,v_n) $ of the maximal segments of new edges associated to the star with center $v_n$.
        \begin{itemize}
            \item If, for infinitely many $n$, $v_n = v = v_\varepsilon$ then $\tilde{r}$ is edge-equivalent to $\tilde{r}_\varepsilon$.
            \item If that is not the case, then $\tilde{r}$ passes through finitely many edges of each center $v_n$.
        \end{itemize}
    \end{itemize}    

    \begin{figure}[h!]
    \centering
    \begin{tikzpicture}
        \def\radius{0.15}
        \def\dy{1}
        \def\dx{1.4}

        \draw (0,0)-- (\dx,0);
        \draw (2*\dx,0)-- (3*\dx,0);
        \draw (4*\dx,0)-- (5*\dx,0);
        \draw[->] (6*\dx,0)-- (7*\dx,0);
        \draw[dashed] (\dx,0)-- (2*\dx,0);
        \draw[dashed] (3*\dx,0)-- (4*\dx,0);
        \draw[dashed] (5*\dx,0)-- (6*\dx,0);
    
        \foreach \i in {1,3,5}{
            \foreach \n in {1,...,10}{
                \pgfmathtruncatemacro{\j}{(\i+1)/2};
                \node[above] at (\i*\dx,\dy) {$v_\j$};
                \node[below] at ({\i*\dx + \dx/2},0) {$p_\j$};
                
                \node[left] at ({\i*\dx},\dy/2) {$e_\j^1$};
                \node[left] at ({(\i+1)*\dx},\dy/2) {$e_\j^2$};
                
                \draw (\i*\dx,\dy) -- ({(\i+1-(1.1)^(-\n))*\dx},0);
                }
        }
    \end{tikzpicture}
\caption{Construction of $(e_n^1,e_n^2)$.}
\end{figure}
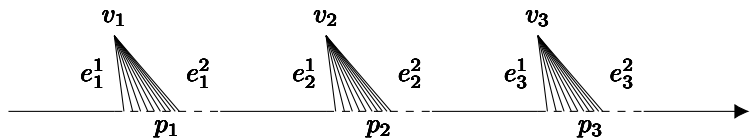
    
    Using this hypothesis we can find pairs of old edges $(e_n^1,e_n^2)$ such that substituting each $P_n$ fragment in the ray for the edge-path $\langle e_n^1,e_n^2 \rangle$ gives us an equivalent $G$-ray of edges.
\end{proof}
\fi 

In the following, we argue that the construction of $\tilde{G}$ is edge-direction preserving, i.e. 

\begin{lemma}
\label{directionpreserving}
    For any graph $G$, $\mathcal{D}_E(G) \approx \mathcal{D}_E(\tilde{G})$ and, therefore, the end spaces of the line graphs of these graphs are homeomorphic.
\end{lemma}

\if \proofshow 1
\begin{proof}
    We shall prove that for every $F \in [\mathrm{E}(G)]^{< \aleph_0}$ there exists $\tilde{F} \in [\mathrm{E}(\tilde{G})]^{< \aleph_0}$ such that the equality\footnote{interpreting the connected components as sets of vertices.} $\mathrm{S}^{E}_{\tilde{G},\tilde{F}} = \mathrm{S}^E_{G,F}$ holds. We of course want $F \subset \tilde{F}$. For each $e \in F$, if $e$ is an edge of a star associated with the vertex $x = v_\varepsilon$, we add to $\tilde{F}$ the edges depicted in \myref{edgefigure}, in other words, $e$'s neighboring new-edges associated with this star. Same argument for the other tip of $e$. This means $|\tilde{F}|\leq 4|F|$ and therefore $\tilde{F}$ is still finite.\\

        \begin{figure}[ht!]
    \begin{minipage}{0.5\textwidth}
    \centering
    \begin{tikzpicture}
        \def\radius{0.15}
        \def\dy{2}
        \def\dx{2.5}
    
        \node[above] at (0,\dy) {$x = v_\varepsilon$};

        \draw (0,\dy) -- (0,0) node[pos = 0.8,right] {$e$};
    
        \foreach \n in {1,...,8}{
            \draw (0,\dy) -- ({(-1 +(1.4)^(-\n))*\dx},0);\draw (0,\dy) -- ({(1 - (1.4)^(-\n))*\dx},0);
        }

        \draw[dashed] (0,0) -- ({(-1 +(1.4)^(-1))*\dx},0) node[midway,below] {$\tilde{e}_e^1$};
        \draw[dashed] (0,0) -- ({(1 -(1.4)^(-1))*\dx},0) node[midway,below] {$\tilde{e}_e^2$};
        \end{tikzpicture}
    \end{minipage}%
    \begin{minipage}{0.5\textwidth}
    \centering
    \begin{tikzpicture}
        \def\radius{0.15}
        \def\dy{2}
        \def\dx{2.5}
    
        \node[above] at (0,\dy) {$ v_\varepsilon$};
    
        \foreach \n in {1,...,8}{
            \draw (0,\dy) -- ({(-1 +(1.4)^(-\n))*\dx},0);\draw (0,\dy) -- ({(1 - (1.4)^(-\n))*\dx},0);
        }

        \draw[dashed]  
            ({(-1 +(1.4)^(-1))*\dx},0)
            --
            ({(1 -(1.4)^(-1))*\dx},0)
            node[midway,below] {$\tilde{e}$};

        \draw (0,\dy) -- ({(-1 +(1.4)^(-1))*\dx},0) node[pos = .8,right] {$e^1_{\tilde{e}}$};
        \node[below] at ({(-1 +(1.4)^(-1))*\dx},0) {$a_i$};
        \draw (0,\dy) -- ({(1 - (1.4)^(-1))*\dx},0) node[pos = .8,left] {$e^2_{\tilde{e}}$};
        \node[below] at ({(1 -(1.4)^(-1))*\dx},0) {$a_{i+1}$};
        \end{tikzpicture}
    \end{minipage}
    \caption{For each $e \in F$ we add $\tilde{e}_{e}^i$ in the situation above to $\tilde{F}$, if they exist. This means that $\tilde{e} \notin \tilde{F}\implies e_{\tilde{e}}^i \notin F$.}
    \label{edgefigure}
    \end{figure}
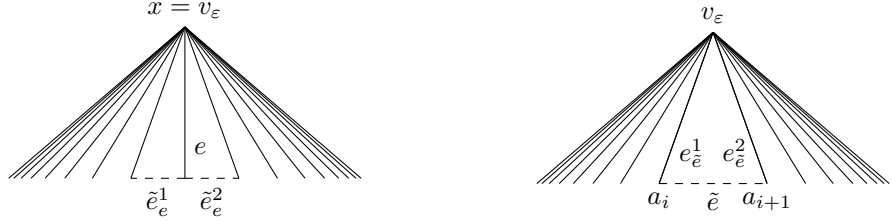

    Let $C$ be $G$-connected in $G \backslash F$, then between any two $u,v \in C$ there is a $G$-walk $P$ between them, but this is\footnote{Remember $G \leq \tilde{G}$.} also a $\tilde{G}$-walk in $G \backslash \tilde{F}$, and we are done. Let $C$ be a $\tilde{G}$-connected component of $\tilde{G} \backslash \tilde{F}$, then between any two $u,v \in C$ there must be a walk $P \doteq \langle u \doteq a_0,a_1,a_2,\dots,a_n \doteq v \rangle$ and the edges between any two contiguous vertices in this walk must be a $\tilde{G}$-edge outside $\tilde{F}$. If all of these edges are in $G$, then $P$ is a $G$-walk with no $F$-edge, which means that $C$ is $G$-connected in $G \backslash F$. Otherwise, take a new-edge $\tilde{e} = \{a_i,a_{i+1}\}$ that appears in $P$, then $\tilde{e}$ was added as $\tilde{e} = \tilde{e}_{\varepsilon,n}$ for some $\varepsilon \in \Omega(G')$ and there is a chosen $v_\varepsilon$. If one of these $e_{\tilde{e}}^i$ were in $F$, then we would have $\tilde{e} \in \tilde{F}$ which is a contradiction. We substitute\footnote{Walks are strings, so we can make these substitutions as long as we respect vertex adjacency.} $\langle a_i,a_{i+1} \rangle$ for $\langle a_i,v_\varepsilon,a_{i+1}\rangle$. Doing this for each new edge in $P$ we arrive at a $G$-walk in $G\backslash F$ testifying $u,v$ are in the same connected component in $G \backslash F$.
\end{proof}
\fi

It is clear that the transition maps will work in the above, and the proof follows the same lines as in \myref{edgedirectionsareendsofedgegraph}. 

\begin{corollary}[Proposition 2.5.10 in \cite{eng}]
\label{directionhomeomorphism}
    For any $G$ the homeomorphism relation $\mathcal{D}_E(G) \approx \mathcal{D}_E(\tilde{G})$ holds, which means that $\Omega(G')\approx \Omega(\tilde{G}')$ under \myref{edgedirectionsareendsofedgegraph}.
\end{corollary}

The association $F \leftrightarrow \tilde{F}$ in the proof of \myref{directionpreserving} associates $[\mathrm{E}(G)]^{<\aleph_0}$ to a cofinal suborder from $[\mathrm{E}(\tilde{G})]^{<\aleph_0}$. In this case, it is not difficult to see that

\begin{corollary}
\label{psicont}
    The map $\psi_G:\mathcal{D}_E(G)\to \Omega_E(\tilde{G}) $ is continuous.
\end{corollary}

\begin{theorem}
\label{completiontheorem}
    Let $G$ be a graph and $\tilde{G}$ its completion. The map $\psi_G$ is a homeomorphism.
\end{theorem}

\begin{proof}
    Observe that $\psi_G:\mathcal{D}_E(G) \to \Omega_E(\tilde{G})$ is surjective by  \myref{directioncompleteness} and continous by \myref{psicont}. We prove that $\psi_G$ is injective. Suppose that $\tilde{\varepsilon} \doteq \psi_G(\rho_1) = \psi_G(\rho_2) = [r]$ is equivalent to an old ray $[r]_{\tilde{G}}$. Then $\rho_1 = \rho_2 = \rho_{[r]}$. Furthermore, if $\tilde{\varepsilon}$ is equivalent to a new ray $[\tilde{r}]$ associated with a star centered at $v$, then $\rho_1 = \rho_2 = \rho_v$. As $\mathcal{D}_E(G)$ is compact Hausdorff, $\psi_G$ is a homeomorphism.
\end{proof}

\sloppy In the previous setting, we guarantee the following homeomorphisms $ \Omega_E(\tilde{G}) \overset{}{\approx} \Omega(G')\overset{}{\approx} \mathcal{D}_E(G) \approx \mathcal{D}_E(\tilde{G})$, so

\begin{corollary}
    For every graph $G$, $G$'s completion graph $\tilde{G}$ is direction-complete. 
\end{corollary}

From the fact that $\mathrm{S}^E_{G,\bullet}$ is always\footnote{It is important here that $G$ itself have finitely many connected components.} finite and \myref{edgedirectionsareendsofedgegraph}, it follows that:

\begin{corollary}\label{COR_EdgeDirCompact}
    For every connected graph $G$, the space $\mathcal{D}_E(G)$ is compact, meaning that the end space of every line-graph is compact.
\end{corollary}

Thus:

\begin{corollary}\label{COR_LineGraphAccumulation}
    Let $G$ be a connected graph. If $S \subset \Omega(G')$ is infinite, then $S$ has an accumulation point.
\end{corollary}

We should highlight the importance of the assumption that $G$ is connected in \myref{COR_EdgeDirCompact}: indeed, a disjoint union of infinitely many rays clearly produces a graph $G$ for which $\mathcal{D}_E(G)$ is infinite and discrete (and, therefore, not compact).

In this case, in order to show that the class of edge-direction spaces is strictly smaller than the class of edge-end spaces, it is not enough to find an edge-end space which is not compact. We nevertheless present such a counterexample in what follows:

\begin{figure}[h!]
    \centering
    \begin{tikzpicture}
        \def\radius{0.15}
        \def\dy{1}
        \def\dx{2.7}

        \foreach \x in {1,...,4}{
            \node (\x) at (\x*\dx+0.9,-\dy+1) {\tiny $\cdots$};
            
            \foreach \n in {0,...,2}{
                \draw[-{Latex[length=1.5mm]}] (\x*\dx,-\dy) -- ({(\x - 1 + (1.25)^(-\n))*\dx},0.8*\dy);
                \draw[-{Latex[length=1.5mm]}] (\x*\dx,-\dy) -- ({(\x + 1 - (1.25)^(-\n))*\dx},0.8*\dy);
    
                \draw (\x*\dx,-\dy) -- (\x*\dx,-1.5*\dy);
            }}

        \foreach \x in {1,...,4}{
            \foreach \n in {2,...,4}{
                \node (r\x\n) at ({(\x + 1 - (1.3)^(-\n+2))*\dx},0.8*\dy+0.2) {\tiny $\varepsilon_{\x}^{\n}$};
            }
            \node (r\x0) at ({(\x - 1 + (1.3)^(-1-1))*\dx},0.8*\dy+0.2) {\tiny $\varepsilon_{\x}^{0}$};
            \node (r\x1) at ({(\x - 1 + (1.3)^(-0-1))*\dx},0.8*\dy+0.2) {\tiny $\varepsilon_{\x}^{1}$};
        }

        \node (dots) at (4*\dx+1,-\dy-0.25) {$\cdots$};    
        \node (r) at (5*\dx,-1.5*\dy) {$\varepsilon_0$};

        \draw[->] (0,-1.5*\dy) -- (r);
    \end{tikzpicture}
    \caption{Figure of the graph given in \myref{notendspaceofedgegraph}.}
    \label{thisexamplefig}
\end{figure}

\begin{example}
    \label{notendspaceofedgegraph}
    Let $X$ be the edge-end space of the graph of \myref{thisexamplefig} -- we claim that it cannot be homeomorphic to $\Omega(G')$ for any $G$. 
    
    Indeed, suppose $G$ is a graph and $f\colon X\to \Omega(G')$ is a homeomorphism. Note that, for any two infinite sets of natural numbers $\set{n_i:i\in\omega}$ and $\set{m_j:j\in\omega}$, $\varepsilon_0$ is an accumulation point of 
    \[\set{\varepsilon_{n_i}^{m_j}: i,j\in\omega}.\]

    Thus, it follows that the end $f(\varepsilon_0)$ in $\Omega(G')$ is (either as a ray or as an infinite star) in a connected component $H$ of $G$ containing the set of ends $\set{f(\varepsilon_n^k):k\in\omega}$ (either as rays or as infinite stars) for cofinally many $n\in\omega$. Fix any such $n\in\omega$. Then note that $\set{\varepsilon_n^k:k\in\omega}$ has no accumulation point. It follows that $\set{f(\varepsilon_n^k):k\in\omega}\subset \Omega(H')\subset \Omega(G')$ is an infinite set with no accumulation point. Since $H$ is connected, this contradicts \myref{COR_LineGraphAccumulation}.
\end{example}

    \section{Compact edge-end spaces}
    As in \myref{THM_CompactCharacterization}, we will now use the inclusion of the edge-end space in its direction space to characterize compactness of the edge-end space topologically. Combining the fact that direction spaces are compact (\myref{COR_EdgeDirCompact}) and Hausdorff with \myref{directionrepresentation} we obtain the following characterization of edge-end compactness:

\begin{corollary}
\label{raylesscharact}
    The following are equivalent for a connected graph $G$:
    \begin{itemize}
        \item[(i)] The edge-end space of a graph $G$ is compact. 
        \item[(ii)] The canonical $\iota_G:\Omega_E(G) \rightarrow \mathcal{D}_E(G)$ includes $\Omega_E(G)$ as a closed subset of $\mathcal{D}_E(G)$.
        \item[(iii)] The set of rayless directions is an open set of $\mathcal{D}_E(G)$.
        \item[(iv)] Every timid vertex of infinite degree $v$ admits a finite edge-set $F$ such that $v \in C \in \mathrm{S}_{G,F}$ and $C$ is rayless.
    \end{itemize}
\end{corollary}

\subsection{Compactness and timid vertices}
\label{covering}

In this section we develop a combinatorial characterization analogue to \myref{THM_CompactCharacterization}. In order to do this, though, we will need to rely on a few constructions previously presented in the literature.

\paragraph{Edge-end spaces represented as end spaces} Given a graph $G$, consider the following graph $H_G$, which is defined in \cite{aurichi2024topologicalremarksendedgeend} after expanding some vertices to cliques. Formally, for each edge-dominating vertex $v\in \mathrm{V}(G)$ we will add a complete graph $K_v \doteq K_{\deg{v}}$ with $\deg(v)$ many vertices. Each $v$-neighbor $u$ corresponds to a vertex $u^v \in\mathrm{V}(K_v)$. Denoting by $D$ the set of vertices of $G$ that edge-dominate some ray, the vertex set of $H_G$ is given by \[\mathrm{V}(H_G) = (\mathrm{V}(G)\setminus D) \cup\displaystyle \bigcup_{v\in D}\mathrm{V}(K_v),\] while its edge set is defined by the following conditions:
\begin{itemize}
    \item[(i)] Add the edge $\{u_1^v,u_2^v\}$ for each pair of neighbors $u_1,u_2$ of the edge-dominating $v$, as to make $K_v$ into a complete sub-graph.
    \item[(ii)] We maintain the $G$ edges $\{u,v\}$ for timid pairs.
    \item[(iii)] Add the edge $\{u,u^v\}$ if $u$ is a timid neighbor of the dominating $v$.
    \item[(iv)] Add the edge $\{v^u,u^v\}$ if $u,v \in D$ are neighboring edge-dominating vertices.
\end{itemize}

Given a graph $G$, note that the construction of the graph $H_G$ is naturally equipped with a function $\theta: E(G) \longrightarrow E(H_G)$ defined as follows:

\begin{itemize}
        \item $\theta(\lbrace u,v\rbrace) = \lbrace u, v\rbrace$ if $u,v \in V(G)\setminus D$;
        \item $\theta(\lbrace u, v\rbrace) = \lbrace u, u^v\rbrace$ if $u\in V(G)\setminus D$ but $v \in D$;
        \item $\theta(\lbrace u, v\rbrace) = \lbrace v^u, u^v\rbrace$ if $u,v \in D$.
    \end{itemize}

This map also translates rays of $G$ into rays of $H_G$ in a natural sense. Indeed, if $r = \langle v_0,v_1,v_2,v_3\dots\rangle$ is a ray in $G$, we consider the ray $\theta(r)$ in $H_G$ whose presentation by its \textit{edges} is $\theta(\lbrace v_0,v_1\rbrace)s_1\theta(\lbrace v_1,v_2\rbrace)s_2\theta(\lbrace v_2, v_3\rbrace)\dots$, where

\begin{itemize}
        \item $s_i =\emptyset$ is the empty edge if $v_i\notin D$, for $i \geq 1$;
        \item $s_i  = \lbrace v_{i-1}^{v_i},v_{i+1}^{v_i}\rbrace$ is the edge in $K_{v_i}$ connecting the canonical edges $\theta(\lbrace v_{i-1}, v_i\rbrace)$ and $\theta(\lbrace v_i,v_{i+1}\rbrace)$, for $i \geq 1$.
    \end{itemize}

 \renewcommand{\myrectangle}[5]{
    \draw (#1,#2) rectangle ++ (#3,#4);
    \node[shift = {({-#3*0.5},{0.1*#4})}] at (#1+#3,#2+#4) {\large\textbf{#5}};
}

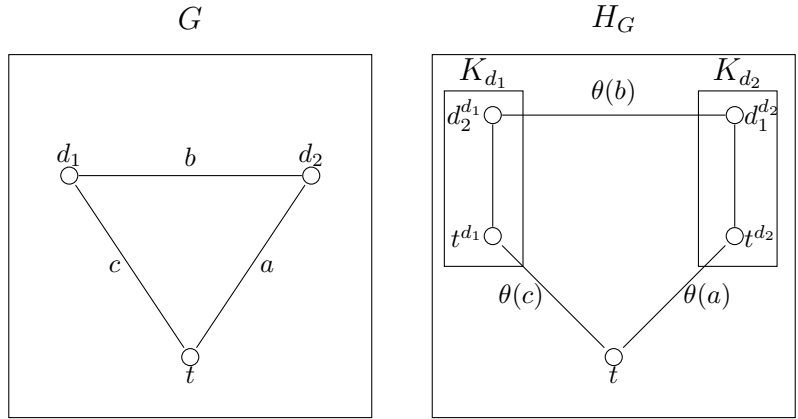
\begin{figure}[ht]
    \centering
    \begin{tikzpicture}
        \def \dx {0.8}
        \def \dy {0.8}

        \myrectangle{0}{0}{6*\dx}{6*\dy}{$G$}

        \node (t) at ({3*\dx},{\dy}) {};
        \node (d1) at ({\dx},{4*\dy}) {};
        \node (d2) at ({5*\dx},{4*\dy}) {};

        \draw (t)--(d1) node[midway,left] {$c$} --(d2) node[midway,above] {$b$} -- (t) node[midway,right] {$a$};
        \draw (t) circle (3.2pt);
        \draw (d1) circle (3.2pt);
        \draw (d2) circle (3.2pt);

        \node[below] at (t) {$t$};
        \node[above] at (d1) {$d_1$};
        \node[above] at (d2) {$d_2$};

        \begin{scope}[shift = {(7*\dx,0)}]
            \myrectangle{0}{0}{6*\dx}{6*\dy}{$H_G$}

        \node (t) at ({3*\dx},{\dy}) {};
        \node (d12) at ({\dx},{5*\dy}) {};
        \node (d21) at ({5*\dx},{5*\dy}) {};
        \node (d1t) at ({\dx},{3*\dy}) {};
        \node (d2t) at ({5*\dx},{3*\dy}) {};

        \draw (t)--(d2t) node[midway,right] {$\theta(a)$} --(d21) -- (d12) node[midway,above] {$\theta(b)$}-- (d1t) -- (t) node[midway,left] {$\theta(c)$};
        \draw (t) circle (3.2pt);
        \draw (d12) circle (3.2pt);
        \draw (d21) circle (3.2pt);
        \draw (d1t) circle (3.2pt);
        \draw (d2t) circle (3.2pt);

        \node[below] at (t) {$t$};
        \node[left] at (d12) {$d_2^{d_1}$};
        \node[right] at (d21) {$d_1^{d_2}$};
        \node[left] at (d1t) {$t^{d_1}$};
        \node[right] at (d2t) {$t^{d_2}$};

        \myrectangle{0.2*\dx}{2.5*\dy}{1.3*\dx}{2.9*\dy}{$K_{d_1}$}

        \begin{scope}[shift = {(4.2*\dx,0)}]
            \myrectangle{0.2*\dx}{2.5*\dy}{1.3*\dx}{2.9*\dy}{$K_{d_2}$}
        \end{scope}
        \end{scope}
    \end{tikzpicture}
    \caption{Construction of $H_G$. In the above, $d_1$ and $d_2$ are dominating, while $t$ is timid.}
    \label{H_Gfig}
\end{figure} 

 Using the map $\theta$, it was shown that
\begin{theorem}[Theorem 2.1 in \cite{aurichi2024topologicalremarksendedgeend}]\label{THM_HGHomeo}
For every graph $G$, $\Omega(H_G)\cong \Omega_E(G)$.
\end{theorem}

\noindent and it is easy to see that

\begin{proposition}\label{PROP_Dominant_FinComp}
    Let $G = (V,E)$ be a connected graph and $e\in E$ be an edge adjacent to a vertex $v$ which edge-dominates some ray in $G$. Then $|\mathrm{S}_{H_G, \{v_e\}}|\le 2$.
\end{proposition}

Now suppose that $F$ is a finite set of timid vertices in $G$. Given a connected component $C\in \tilde{\mathrm{S}}_{G,F}$ with $v\in C$, let $C_H\in \tilde{\mathrm{S}}_{H_G,F}$ be the connected component containing $v$ if $v$ is timid, or, otherwise, containing $v_e$ for some edge $e\in C$ adjacent to $v$. In this case:

\begin{lemma}\label{LEMMA_Compon_Bijec}
    The association 
    \begin{align*}
        \varphi\colon \: \tilde{\mathrm{S}}_{G,F}\,&\to \,\tilde{\mathrm{S}}_{H_G,F}\\
        C&\mapsto C_H
    \end{align*}
    is a well defined bijection.
\end{lemma}
\begin{proof}
    Suppose $v$ and $v'$ are two vertices in $C\in \tilde{\mathrm{S}}_{G,F}$. Then there is a $G$-path $P$ in $C$ from $v$ to $v'$ which avoids $F$. Hence, such path is translated by $\theta$ to an $H$-path $P_H$ in $H_G$ which avoids $F$ as well, so the connected component $C_H$ obtained from $v$ is the same which is obtained from $v'$.\\

    To show that $\varphi$ is surjective, let $\tilde{C}\in \tilde{\mathrm{S}}_{H_G,F}$. Then there is a ray $r$ in $G$ such that $\theta (r)$ is in $\tilde{C}$. Let $C\in \tilde{\mathrm{S}}_{G,F}$ be the connected component containing a tail of $r$. Then it is clear that $C_H = \tilde{C}$.\\

    At last, suppose $C_H = C_H'$ for some $C, C'\in \tilde{\mathrm{S}}_{G,F}$. Without loss of generality, suppose a timid $v\in C$ and an edge $e\in C'$ adjacent to an edge-dominant vertex $v'\in C'$. Then there is an $H$-path $P_H$ in $C_H$ starting at $v$ and ending at $v_e$. Note that $\theta^{-1}(P_H)=P$ is a $G$-path connecting $v$ to $v'$, which shows that $C = C'$. Hence, $\varphi$ is injective and the proof is complete.
\end{proof}

We obtained all the tools to show that
\begin{theorem}
\label{compactnesscharacterization}
    The edge-end space of a graph $G$ is compact if, and only if, for every finite set of timid vertices $F$, the collection $\tilde{\mathrm{S}}_{G,F}$ of non-rayless connected components is finite.
\end{theorem}
\begin{proof}
    Suppose $\Omega_E(G)$ is compact and let $F$ be a given finite set of timid vertices. By \myref{THM_HGHomeo}, so is $\Omega(H_G)$. Thus, by \myref{THM_CompactCharacterization}, $\tilde{\mathrm{S}}_{H_G,F}$ is finite. In this case, it follows from \myref{LEMMA_Compon_Bijec} that $\tilde{\mathrm{S}}_{G,F}$ is also finite.\\

    Now assume $\tilde{\mathrm{S}}_{G,F}$ is finite for every finite set of timid vertices $F$ in $G$. We will show that $\tilde{\mathrm{S}}_{H_G,F}$ is finite for every finite $F\subset \mathrm{V}(H_G)$, which will conclude the proof in view of  \myref{THM_CompactCharacterization} and \myref{THM_HGHomeo}. So let such $F\subset \mathrm{V}(H_G)$ be given. Consider 
    \[F' = F\cap \mathrm{t}(G).\]

    Then $\tilde{\mathrm{S}}_{G,F'}$ is finite and, by  \myref{LEMMA_Compon_Bijec}, so is $\tilde{\mathrm{S}}_{H_G,F'}$. Furthermore, it follows from \myref{PROP_Dominant_FinComp} that the additional $F\setminus F'$ vertices can only separate the finite components of $\tilde{\mathrm{S}}_{H_G,F'}$ into more finite components in $\tilde{\mathrm{S}}_{H_G,F}$, so the proof is complete.
\end{proof}

\subsection{Compact edge-end spaces as edge-direction spaces}\label{SUBSEC_CompactEdgeEndIsEdgeDirection}

\paragraph{}
While we do know from \myref{notendspaceofedgegraph} that some edge-end spaces cannot be represented as edge-direction spaces, it should be noted that the counter-example space presented in \myref{notendspaceofedgegraph} is not compact -- and this fact was crucial in showing that it is indeed a counter-example. 

Thus, since the edge-direction space of any connected graph is always compact (as stated in \myref{COR_EdgeDirCompact}), it is only natural to ask whether the additional compactness hypothesis is sufficient to give us such a representation theorem. We answer this question positively in \myref{THM_CompectEdgeEndDirection}, but we will need some technical results before we get to that point. 

Given $u,v\in \mathrm{V}(G)$, we write 
\begin{equation}\label{EQ_VertexEdgeEqv}
    u\sim_E v
\end{equation} 
if, for every finite $F\subset \mathrm{E}(G)$, $u$ and $v$ share the same connected component in $G\setminus F$. Note that $v\in \mathrm{t}(G)$ and $u\sim_E v$ implies that $u\in \mathrm{t}(G)$. 

It will help us in the construction of the connected graph $H$ of \ref{THM_CompectEdgeEndDirection} that every timid vertex in $G$ is $\sim_E$-equivalent only to itself. In this case, let us denote by $\sim$ the equivalence relation over $\mathrm{V}(G)$ such that 
\[
v\sim u \iff \begin{cases}
                v\sim_E u, \text{ if $v\in \mathrm{t}(G)$,}\\
                v=u, \text{ otherwise.}
            \end{cases}
\]

Given $v\in \mathrm{V}(G)$, we let
\[ [v] \doteq \set{u\in \mathrm{V}(G): u\sim v}\]
and then consider the graph denoted by $G/_{\sim}$, in which 
\begin{align*}
 \mathrm{V}(G/_{\sim}) =& \set{[v]:v\in \mathrm{V}(G)}, \\
 \{[v],[u]\}\in \mathrm{E}(G/_{\sim}) \iff& \exists v'\sim v, \,  u'\sim u\left(\{v',u'\}\in \mathrm{E}(G)\right)
\end{align*}
for all $u,v\in \mathrm{V}(G)$. 

A vertex set $U \subset \mathrm{V}(G)$ is \emph{dispersed} when there is no comb in $G$ with infinitely many teeth in $U$. Note that, for any timid $v$, the vertex set $[v]$ is dispersed: 
Suppose there is a comb with infinitely many teeth $d_i \in [v]$ and spine $r$ and that $F\subset \mathrm{E}$ is finite. Then there must be an $i\in \mathbb{N}$ such that the path from $d_i$ to the $(G\setminus F)$-tail of $r$ does not pass through $F$. Furthermore, since $v\sim v'$, there must be a $(G\setminus F)$-path from $v$ to $v'$, so that $v$ is in the same connected component as the $(G\setminus F)$-tail of $r$ and thus $v$ edge-dominates $r$, a contradiction.

Now let $\pi\colon \mathrm{V}(G) \to \mathrm{V}(G/_{\sim})$ be the quotient projection (i.e., $\pi(v) = [v]$ for every $v\in \mathrm{V}(G)$). Then it is easy to see that, given $e\in \mathrm{E}(G/_{\sim})$, the set
\[\pi^{-1}e \doteq \set{\{x,y\}\in \mathrm{E}(G): e = \{\pi(x),\pi(y)\}}\]
is finite. This gives us:

\begin{lemma}\label{LEMMA_RayProjection}
    For every ray $r$ in $G$ there is a ray $r^\pi$ in $G/_{\sim}$ contained in $\pi[\mathrm{V}(r)]$. Moreover, if $r'$ is another ray which passes through infinitely many vertices of $\pi[\mathrm{V}(r)]$, then $r'\sim_{E(G/_{\sim})}r^\pi$.
\end{lemma}
\begin{proof}
    Suppose $r=\seq{v_n:n\in\omega}$ is a ray in $G$. From the observation discussed above, there must be an $n_0\in\omega$ (possibly equal to $0$) such that $v_0\sim v_{n_0}$ and $v_0\nsim v_n$ for every $n\ge n_0$. For the same reason, there must be an $n_1>n_0$ (possibly equal to $n_0+1$) such that $v_{n_0+1}\sim v_{n_1}$ and $v_{n_0+1}\nsim v_n$ for every $n\ge n_1$. By proceeding in this manner we clearly construct a ray $r^\pi = \seq{[v_{n_k}]:k\geq 0}$ in $G/_{\sim}$ which is contained in $\pi[\mathrm{V}(r)]$.

    Now assume that $r'=\seq{[u_n]:n\in\omega}$ is a ray passing through infinitely many vertices in $\pi[\mathrm{V}(r)]$ and a finite $F\subset E(G/_{\sim})$ is given. Since $\pi^{-1} F\doteq \set{\pi^{-1}e:e\in F}$ is finite, there must be an $N\in\omega$ such that $\mathrm{E}(\seq{v_n:n\ge N})\cap \pi^{-1}F = \emptyset$. In this case, fix $n_k\in\omega$ such that $n_k>N$. Since $r'$ intersects $\pi[\mathrm{V}(r)]$ in infinitely many vertices, there must be an $M>n_k$ such that $[v_M]\in \mathrm{V}(r')$ and the tail of $r'$ from $M$ onward avoids $F$. Hence, $\seqq{[v_{n_k}], [v_{n_k+1}], \dotsc, [v_M]}$ attests that the tails of $r^\pi$ and $r'$ in $(G/_{\sim})\setminus F$ are in the same connected component of $(G/_{\sim})\setminus F$.
\end{proof}

Thus a map $\bar{\pi}\colon \Omega_E(G)\to \Omega_E(G/_{\sim})$ such that $\pi([r]_E) = [r^\pi]_{E(G/_{\sim})}$, where $r^\pi$ is a ray contained in $\pi[\mathrm{V}(r)]$ (as in \myref{LEMMA_RayProjection}) is well-defined. 

We can obtain another corollary of \myref{LEMMA_RayProjection}):

\begin{corollary}
     If $v\in \mathrm{V}(G)$ is edge-dominant in $G$, then $[v]=\{v\}$ is edge-dominant in $G/_{\sim}$.
\end{corollary}
\begin{proof}
    Suppose $v\in \mathrm{V}(G)$ edge-dominates a ray $r$ in $G$. We claim that $[v] = \{v\}$ edge-dominates the ray $r^\pi$ given by \myref{LEMMA_RayProjection}. Indeed, suppose a finite $F\subset E(G/_{\sim})$ is given. Since $\pi^{-1}F\subset \mathrm{E}(G)$ is finite, we can find a $G$-path $P$ from $v$ to the tail of $r$ in $G\setminus \pi^{-1}F$ which avoids $\pi^{-1}F$ (by extending $P$ along $r$ if necessary, we may assume that $P$ ends in a vertex $v_n$ such that $[v_n]\in r^\pi$). Thus, $\pi[\mathrm{V}(P)]$ is a $(G/_{\sim})$-walk attesting that $[v]$ and a tail of $r^\pi$ are in the same connected component of $(G/_{\sim})\setminus F$. 
\end{proof}

The following Lemma (as well as its consequences) will be crucial in the proof that $\bar{\pi}$ is always a homeomorphism.

\begin{lemma}\label{LEMMA_PiVertexSeparation}
    If $F\subset \mathrm{E}(G)$ separates $v$ from $u$ in $G$ and $F =\pi^{-1}(\pi[F])$, where 
    \[\pi[F] \doteq \set{\{[x],[y]\}:\{x,y\}\in F},\]
    then $\pi[F]$ separates $[v]$ from $[u]$ in $G/_{\sim}$.
\end{lemma}
\begin{proof}
    \sloppy Let $P = \seqq{[v], [w_1], [w_2],\dotsc, [w_n], [u]}$ be a $(G/_{\sim})$-path and consider the set $\pi^{-1}\mathrm{E}(P)\subset\mathrm{E}(G)$. 
    
    Fix $e_0=\{v',w_1'\}\in \pi^{-1}(\{[v], [w_1]\})\in \mathrm{E}(G)$. Then $v'\sim v$ and $w_1'\sim w_1$, so there are $G$-paths $P_0$ from $v$ to $v'$ and $Q_0$ from $w_1'$ to $w_1$ which avoid $F$. Now fix  $e_1=\{w_1'',w_2'\}\in \pi^{-1}(\{[w_1], [w_2]\})\in \mathrm{E}(G)$. Then $w_1''\sim w_1$ and $w_2'\sim w_2$, so there are $G$-paths $P_1$ from $w_1$ to $w_1''$ and $Q_1$ from $w_2'$ to $w_2$ which avoid $F$.

    By concatenating the obtained $G$-paths $P_0,Q_0,P_1, Q_1, \dotsc, P_n, Q_n$ we build a $G$-path $P'$ from $v$ to $u$ in $G$. Since $F$ separates the two vertices, $P'$ cannot avoid $F$. But the $G$-paths $P_0,Q_0, P_1, Q_1, \dotsc, P_n, Q_n$ do avoid $F$ by construction, so it follows that $e_i\in F$ for some $i\le n$. We thus conclude that $P$ does not avoid $\pi[F]$ (since $\pi e_i \in \mathrm{E}(P)\cap \pi[F]$).
\end{proof}

\begin{corollary}\label{LEMMA_PiRaySeparation}
    If $F\subset \mathrm{E}(G)$ separates $r_0$ from $r_1$ in $G$, $F =\pi^{-1}(\pi[F])$ and $r_0^\pi,r_1^\pi$ are rays of $G/_{\sim}$ as in \myref{LEMMA_RayProjection}, then $\pi[F]$ separates $r_0^\pi$ from $r_1^\pi$ in $G/_{\sim}$.
\end{corollary}

\begin{corollary}
    For every $[v]\in \mathrm{t}(G/_{\sim})$ and $u\in \mathrm{V}(G/_{\sim})$, 
    \[
        [v]\sim_{E(G/_{\sim})}[u] \implies [v]=[u].
    \]
\end{corollary}

Finally, the last ingredient for showing the upcoming \myref{THM_QuotientHomeo} is the following

\begin{lemma}\label{LEMMA_JustCombLemma}
    Suppose $G$ is graph, $\seq{v_n:n\in\omega}$ and $\seq{F_n:n\in\omega}$ are sequences such that, for every $n\in\omega$, 
    \begin{itemize}
        \item[(i)] The set of edges $F_n\subset \mathrm{E}(G)$ is finite;
        \item[(ii)] The set $F_n$ separates the vertex $v_{n+1}$ from $v_k$ for every $k\le n$;
        \item[(iii)] $F_n$ does not separate $v_{k}$ from $v_{m}$ for any $k>m>n$.
    \end{itemize}
     Then there is an infinite comb in $G$ with infinitely many teeth in $\set{v_n:n\in\omega}$.
\end{lemma}
\begin{proof}
    For each $n\in\omega$, let $C_n$ be the connected component of $v_n$ in $G\setminus F_n$. Let $H$ be the induced sub-graph of $G$ such that 
    \[\mathrm{V}(H) = \bigcup_{n\ge 1}C_n.\]

    Then condition (iii) implies that $H$ is connected. Furthermore, conditions (ii) and (iii) together tell us that there can be no infinite star inside $H$ with tips in $\set{v_n:n\in\omega}$. In this case, we can apply \myref{starcomb} to find the desired infinite comb.
\end{proof}

\begin{theorem}\label{THM_QuotientHomeo}
    For every graph $G$, $\bar{\pi}\colon \Omega_E(G) \to \Omega_E(G/_{\sim})$ is a homeomorphism.
\end{theorem}
\begin{proof}
    \begin{description}
        \item[$\bar{\pi}$ is injective:] It follows from \myref{LEMMA_PiRaySeparation} that if $r_0\nsim_{\mathrm{E}(G)} r_1$ in $G$, then $r_0'\nsim_{\mathrm{E}(G/_{\sim})} r_1'$, so 
        \[ 
            \bar{\pi}([r_0]_{\mathrm{E}(G)}) =  [r_0']_{\mathrm{E}(G/_{\sim})}\neq [r_1']_{\mathrm{E}(G/_{\sim})} = \bar{\pi}([r_1]_{\mathrm{E}(G)}).
        \]
    
        \item[{$\bar{\pi}$ is surjective:}] Let $r=[v_0],[v_1],[v_2],\dotsc$ be a ray in $G/_{\sim}$. 
        We separate the proof in two cases:
        \begin{itemize}
            \item Suppose first that there is an infinite $A \subset B$ such that, for every $n\in A$ and finite $F\subset \mathrm{E}(G/_{\sim})$, there are infinitely many $k\in A$ such that $[v_n]$ and $[v_k]$ are in the same connected component of $(G/_{\sim})\setminus F$. 
            
            Then we construct the sequences $\seq{v_{n_k}:k\in\omega}$ and $\seq{F_k:k\in\omega}$ in $G$ as in \myref{LEMMA_JustCombLemma} as it follows: start by fixing $n_0,n_1\in A$ with $n_0<n_1$. Since $v_{n_0}\nsim_{\mathrm{E}(G)}v_{n_1}$, we can find a finite $F_0\subset \mathrm{E}(G)$ separating $v_{n_0}$ from $v_{n_1}$.
            
            Suppose $\seq{v_{n_k}:k\le m}\in\mathrm{V}(G)^{m}$ and $\seq{F_k:k< m}$ are defined.  Let $C$ be the connected component of $G\setminus \bigcup_{k\le m}F_k$ containing $v_{n_m}$. By our assumption about $A$, the set 
            \[
                A' = \set{n\in A: v_n \in C}
            \]
            is infinite, so we may fix $n_{m+1}\in A'$ such that $n_{m+1}>n_m$, so that $v_{n_{m+1}}\in C$. Since $v_{n_{m+1}}\nsim_{\mathrm{E}(G)}v_{n_k}$ for every $k\le m$, we can find a finite $F_m\subset \mathrm{E}(G)$ separating $v_{n_{m+1}}$ from $v_{n_k}$ for every $k\le m$, which concludes our recursion.
    
            In this case, the spine of the infinite comb given by \myref{LEMMA_JustCombLemma} is clearly a ray $r_G$ in $G$ such that $r_G^\pi$ as given by \myref{LEMMA_RayProjection} is edge-equivalent to $r$ in $G/_{\sim}$.
    
            \item Otherwise, we may assume that each $[v_n]$ can be separated by a finite $F_n\subset \mathrm{E}(G/_{\sim})$ from $[v_k]$ for every $k\neq n$ (without loss of generality, we may assume that such $F_n$ is the minimal set of edges with such property, so that $F_n\setminus e$ does not separate $[v_n]$ from some other $[v_k]$). 
    
            We claim that the pair of sequences $\seq{v_n:n\in\omega}$ and $\seq{\pi^{-1}(F_n):n\in\omega}$ satisfies the conditions of \myref{LEMMA_JustCombLemma}. Indeed, conditions (i) and (ii) are obviously satisfied. Now note that $F_n$ cannot contain $\{[v_k], [v_k+1]\}$ for any $k>n$ (since this would contradict $F_n$'s minimality). Thus, it follows that, for all $k>m>n$, $F_n$ does not contain any edge from the $(G/_\sim)$-path $\seqq{[v_m], [v_{m+1}], \dotsc, [v_k]}$, so $v_{k}$ cannot be separated from $v_{m}$ by $\pi^{-1}(F_n)$.
            
            Once again, the spine of the infinite comb given by \myref{LEMMA_JustCombLemma} is clearly a ray $r_G$ in $G$ such that $r_G^\pi$ as given by \myref{LEMMA_RayProjection} is edge-equivalent to $r$ in $G/_{\sim}$.
        \end{itemize}

        \item[{$\bar{\pi}$ is continuous and open:}] Suppose a finite $F\subset \mathrm{E}(G)$ such that $F =\pi^{-1}(\pi[F])$ and a ray $r$ in $G$ are given. We claim that 
        \[
            \bar{\pi}[\Omega_E(F,[r]_{\mathrm{E}(G)})] = \Omega_E(\pi[F],[r']_{\mathrm{E}(G/_{\sim})}),
        \]
        which concludes the proof in view of $\bar{\pi}$'s surjectivity. Indeed, suppose that $\varepsilon\in \Omega_E(G/_{\sim})\setminus \bar{\pi}[\Omega_E(F,[r]_{\mathrm{E}(G)})]$. By $\bar{\pi}$'s surjectivity, there is some ray $r_0$ in $G$ such that $\pi([r_0]_{\mathrm{E}(G)}) = [r_0']_{\mathrm{E}(G/_{\sim})} = \varepsilon$. Since $\varepsilon \notin \bar{\pi}[\Omega_E(F,[r]_{\mathrm{E}(G)})]$, $F$ separates $r_0$ from $r$. It thus follows from \myref{LEMMA_PiRaySeparation} that $\varepsilon = [r_0']_{\mathrm{E}(G/_{\sim})}\notin \Omega_E(\pi[F],[r']_{\mathrm{E}(G/_{\sim})})$.
    
        Now suppose that $\varepsilon\in \bar{\pi}[\Omega_E(F,[r]_{\mathrm{E}(G)})]$. Then again, by $\bar{\pi}$'s surjectivity, there is some ray $r_0$ in $G$ such that $\pi([r_0]_{\mathrm{E}(G)}) = [r_0']_{\mathrm{E}(G/_{\sim})} = \varepsilon$ and $F$ does not separate $r$ from $r_0$. Thus, if $P$ is a $G$-path from the tail of $r$ to the tail of $r_0$ in $G\setminus F$, then $\pi[P]$ will be a $(G/_\sim)$-walk attesting that the tail of $r'$ and the tail of $r_0'$ in $(G/_{\sim})\setminus\pi[F]$ are in the same connected component. Thus, $\varepsilon = \pi([r_0]_{\mathrm{E}(G)}) = [r_0']_{\mathrm{E}(G/_{\sim})} \in  \Omega_E(\pi[F],[r']_{\mathrm{E}(G/_{\sim})})$.
    \end{description}
\end{proof}

\myref{THM_QuotientHomeo} will be useful because:
\begin{lemma}\label{LEMMA_AloneTimid}
    Suppose that $G$ is a graph and $v\in \mathrm{t}(G)$ is only $\sim_E$-equivalent to itself. Then $v$ has finitely many neighbors in each connected component of $G\setminus \{v\}$.
\end{lemma}
\begin{proof}
    We present a proof for the contrapositive. If $C$ is a connected component of $G\setminus \{v\}$ and the set $\mathrm{N}(v,C)$ of its neighbors in $C$ is infinite, then \myref{starcomb} tells us that we can find in $C$ either a star with infinitely many tips in $\mathrm{N}(v,C)$ or an infinite comb with infinitely many teeth in $\mathrm{N}(v,C)$. But the center of any such star would be a vertex different from $v$ which is $\sim_E$-equivalent to $v$ and the spine of any such comb would be a ray dominated by $v$. Thus, $v$ is either $\sim_E$-equivalent to some other vertex in $G$ or it is not timid.
\end{proof}

\begin{lemma}\label{LEMMA_DenseHConstruction}
    Suppose $G$ is a connected graph such that every $v\in \mathrm{t}(G)$ is $\sim_E$-equivalent only to itself. Then there is a connected induced subgraph $H\subset G$ such that
    \begin{itemize}
        \item[(a)] $G\setminus H$ is rayless;
        \item[(b)] every connected component of $G\setminus H$ is separated from $H$ by a single vertex in $H$;
        \item[(c)] if $v\in \mathrm{t}(H)$ has infinite degree, then all but at most one connected component of $H\setminus \{v\}$ is non-rayless.
    \end{itemize}
\end{lemma}
\begin{proof}
    If $G$ has no timid vertex with infinite degree, then $H=G$ satisfies (a)--(c). So suppose that there is some vertex $v_0\in \mathrm{t}(G)$ with infinite degree. 
    
    We are going to construct a sequence $\seq{(H_n,G_n):n\in\omega}$ such that, for all $n,m\in\omega$, $H_n,G_n\subset\mathrm{V}(G)$ and
    \begin{itemize}
        \item[(i)] $H_n\cap G_m=\emptyset$;
        \item[(ii)] $G_n$ is rayless;
        \item[(iii)] every connected component of $G_n$ in $G$ has a unique neighbor in $G\setminus G_n$ and such neighbor is in $H_n$;
        \item[(iv)] if $u\in \mathrm{V}(G)$ is a neighbor in $G$ of some $v\in H_n$, then $v\in H_k\cup G_k$ for some $k\le n+1$;
        \item[(v)] the subgraph induced by $\bigcup_{k\le n}H_k$ in $G$ is connected;
        \item[(vi)] if $v\in H_n\cap \mathrm{t}(G)$ has infinite degree, then $G_{n+1}$ contains all but (at most) a single rayless connected component of $G\setminus\{v\}$.
    \end{itemize}

    Start by letting $H_0 = \{v_0\}$. In this case, let $G_0$ be the union of the connected components in $G\setminus \{v_0\}$ which are rayless. Then it is clear that $\langle (H_0,G_0)\rangle$ satisfies (i)--(vi). 

    Now suppose that $\seq{H_k:k\le n}$ and $\seq{G_k:k\le n}$ are defined. Let $H_{n+1}\subset \mathrm{V}(G)$ be the set of vertices which have some neighbor in $H_n$ and are not in $G_k$ for any $k\le n$. If there is no timid vertex with infinite degree in $H_{n+1}$, let $G_{n+1} = \emptyset$. Otherwise, for each $v\in H_{n+1}\cap \mathrm{t}(G)$ with infinite degree, let $G_{v}$ be the union of the connected components in $G\setminus \{v\}$ which are rayless, \emph{except for the one containing $v_0$, if there is such one}. At last, let $G_{n+1}$ be the union of such $G_v$'s.

    It is clear that (i)--(vi) are satisfied for $\seq{(H_k,G_k):k\le n+1}$.

    Let $H$ be the subgraph of $G$ induced by $\mathrm{V}(H) = \bigcup_{n\in\omega}H_n$. 
    
    \begin{description}
        \item[(a) and (b):] Suppose that $C$ is a connected component of $G\setminus H$. Note that, by (i) and (iv), $H = G\setminus \bigcup_{n\in\omega}G_n$. In this case, if $u\in C$, then there is an $n\in\omega$ such that $u\in G_n$. Since $G_n\cap \mathrm{V}(H)=\emptyset$, it follows that $C$ is a connected component of $G_n$ in $G$, which concludes the proof in view of (ii) and (iii). 
    \end{description}
        
    Note that, if $u,v\in \mathrm{V}(H)$ and $P$ is a $G$-path from $u$ to $v$, then (b) tells us that $P$ must be an $H$-path (for, if $P$ otherwise left $H$ at some point, it could only return to $H$ by passing through the same vertex from which it left, contradicting the assumption that $P$ is a path). 
        
    \begin{description}
        \item[(c):] Let $v\in \mathrm{t}(H)$. We claim that $v\in \mathrm{t}(G)$: indeed, suppose that $v\in \mathrm{V}(H)\setminus\mathrm{t}(G)$ and fix a ray $r$ in $G$ which is edge-dominated in $G$ by $v$. Note that (a) tells us that $H$ must contain some vertex of $r$ and (b) tells us that $r$ cannot leave $H$ from that point onward, so $H$ contains a tail of said $r$. 
        
        Let a finite $F\subset \mathrm{E}(H)$ be given. Since $F\subset \mathrm{E}(G)$, there must be some $G$-path $P$ connecting $v$ to the tail of $r$ in $G\setminus (F\cup F')$, where $F'$ is the (finite) set of edges of $r$ which are not in $H$. Our previous observation allows us to conclude that $P$ must be contained in $H$. Thus, $v$ dominates the tail of $r$ contained in $H$.

        The desired conclusion is thus reached in view of (iv) and (vi).
    \end{description}
\end{proof}

The proof of \myref{LEMMA_DenseHConstruction} shows that \myref{LEMMA_DenseHConstruction}(b) implies that $H$ is closed by $G$-paths and, hence, every $G$-ray $r$ has a tail in $H$ (let us denote its maximal tail contained in $H$ as $\restrict{r}{H}$). This fact will be useful for showing:

\begin{corollary}\label{COR_DenseEdgesOnDirections}
    Let $G$ be a graph. Then there is a connected graph $H$ such that $\Omega_E(G)\approx \Omega_E(H)$ and the cannonical embedding $\iota\colon \Omega_E(H)\to \mathcal{D}_E(H)$ has dense image.
\end{corollary}
\begin{proof}
    We can assume that $G$ is connected (for we can add a vertex to $G$ which is adjacent to exactly one vertex from each connected component of $G$, otherwise). Furthermore, in view of \myref{THM_QuotientHomeo}, we may assume that every $v\in \mathrm{t}(G)$ is $\sim_E$-equivalent only to itself. Thus, we may let $H$ be as in \myref{LEMMA_DenseHConstruction}. 

    Let us start by finding the homeomorphism between $\Omega_E(G)$ and $\Omega_E(H)$. In order to do this, note that if $r$ and $r'$ are two rays contained in $G$ which are edge-equivalent in $G$, then there are infinitely many pairwise edge-disjoint $G$-paths attesting such edge-equivalence in $G$. But \myref{THM_QuotientHomeo}(b) tells us that the infinitely many of these paths which connect $\restrict{r}{H}$ to $\restrict{r'}{H}$ must be contained in $H$, so $\restrict{r}{H}$ and $\restrict{r'}{H}$ are edge-equivalent in $H$.

    Thus, the map $\psi\colon \Omega_E(G) \to \Omega_E(H)$ such that $\psi([r]_{\mathrm{E}(G)}) = [\restrict{r}{H}]_{\mathrm{E}(H)}$ is well defined. Since $H$ is an induced subgraph of $G$, it is easy to see that $\psi$ is a bijection. We claim that for every finite $F\subset \mathrm{E}(G)$ and ray $r$ in $G$,
    \begin{equation}\label{EQ_CompactHHomeo0}
      \psi[\Omega_E(F,[r]_{\mathrm{E}(G)})] = \Omega_E(F\cap H,[\restrict{r}{H}]_{\mathrm{E}(H)}).  
    \end{equation}

    Indeed, suppose that $r'$ is a ray contained in the same component as $r$ in $G\setminus F$. Then $\restrict{r'}{H}$ is also in the same component as $\restrict{r}{H}$ in $H\setminus (F\cap H)$. By \myref{THM_QuotientHomeo}(b), a $G$-path between $\restrict{r}{H}$ and $\restrict{r'}{H}$ in $G\setminus F$ must also be contained in $H$. Thus,
    \[
        \psi([r']_{\mathrm{E}(G)}) = [\restrict{r'}{H}]_{\mathrm{E}(H)}\in \Omega_E(F\cap H,[\restrict{r}{H}]_{\mathrm{E}(H)}).
    \]

    Now suppose that $r'$ is a ray in $G$ such that $[\restrict{r'}{H}]_{\mathrm{E}(H)}\in \Omega_E(F\cap H,[\restrict{r}{H}]_{\mathrm{E}(H)})$.
    Then there is an $H$-path $P$ connecting $\restrict{r'}{H}$ to $\restrict{r}{H}$ which avoids $F\cap H$. Since $P$ is contained in $H$, it follows that $P$ also avoids $F$. Hence, $P$ attests that $r'$ and $r$ are in the same component of $G\setminus F$ and thus
    \[
        \psi([r']_{\mathrm{E}(G)}) = [\restrict{r'}{H}]_{\mathrm{E}(H)}\in \psi[\Omega_E(F,[r]_{\mathrm{E}(G)})].
    \]

    It is clear that \myref{EQ_CompactHHomeo0} implies that $\psi$ is open. Furthermore, if a finite $F\subset \mathrm{E}(H)$ and an $H$-ray $r$ are given, \myref{EQ_CompactHHomeo0} together with the fact that $\psi$ is a bijection tell us that 
    \[
        \psi^{-1}(\Omega_E(F,[r]_{\mathrm{E}(H)})) = \Omega_E(F,[r]_{\mathrm{E}(G)}),
    \]
    which concludes the proof that $\psi$ is a homeomorphism.

    At last, in order to show that $\iota\colon \Omega_E(H)\to \mathcal{D}_E(H)$ has dense image, it suffices to show that $\rho(F)$ is not rayless for every rayless direction $\rho\in\mathcal{D}_E(H)$ and finite $F\subset \mathrm{E}(G)$. So suppose that such rayless direction $\rho\in \mathcal{D}_E(H)$ and finite $F\subset \mathrm{E}(G)$ are given. By \myref{directionrepresentation}, there is a $v\in \mathrm{t}(H)$ of infinite degree such that $\rho = \rho_v$. In this case, by \myref{LEMMA_AloneTimid}, $H\setminus \{v\}$ must have infinitely many connected components and \myref{LEMMA_DenseHConstruction}(c) tells us that all but at most one of these connected components are non-rayless. Thus, there must be a non-rayless connected component $C$ of $H\setminus\{v\}$ with an edge $e$ in which $v$ is incident such that $e\notin F$ and $F\cap\mathrm{E}(C)=\emptyset$. Hence, $\rho(F) = \rho_v(F)\supset C$ is non-rayless.
\end{proof}

As a direct consequence, we get:

\begin{theorem}\label{THM_CompectEdgeEndDirection}
    The edge-end space of a graph $G$ is compact if, and only if, there is some connected graph $H$ such that $\Omega_E(G)\approx\Omega_E(H)\approx \mathcal{D}_E(H)$.
\end{theorem}
\begin{proof}
    It follows from \myref{COR_EdgeDirCompact} that if such an $H$ exists, then $\Omega_E(G)$ is compact.
    
    Now let $H$ be as in \myref{COR_DenseEdgesOnDirections}. Then, by \myref{COR_DenseEdgesOnDirections}, $\Omega_E(H)$ is compact and, thus, $\iota\colon \Omega_E(H)\to \mathcal{D}_E(H)$ must be a homeomorphism.
\end{proof}

\begin{rmk}\label{rmk1}
    Note that $\lbrace \Omega_E(G): G \text{ is a graph and } \Omega_E(G) \text{ is compact} \rbrace \subsetneq \mathcal{D}_E$. Indeed, consider $ G = \bigcup_{i \in \omega} R_i$ as the disjoint union of countable rays. Note that $ \Omega_E(G) = \mathcal{D}_E(G) $ is the discrete topological space $\omega$. Since $ \Omega_E(G) = \mathcal{D}_E(G) $ is not compact, it follows that $ \lbrace \Omega_E(G) : G \text{ is a graph and } \Omega_E(G) \text{ is compact}\rbrace \subsetneq \mathcal{D}_E $.
 
\end{rmk}

    
    \section{Timid-direction and timid-end spaces}
    \myref{THM_CompactCharacterization} characterizes compactness of an end space $\Omega(G)$ in terms of finite sets of vertices, which are the finite sets that are also used to define the ends and basic open sets of $\Omega(G)$. The same does not happen in \myref{compactnesscharacterization}: even though it can be seen as an analogue of \myref{THM_CompactCharacterization}, we use finite sets of edges to define the edge-ends and basic open sets in $\Omega_E(G)$, while the finite sets used to characterize compactness of these spaces are those of \emph{timid vertices}.

It is easy to check, however, that a finite $F\subset \mathrm{t}(G)$ cannot separate the $(G\setminus F)$-tails of two rays which are edge equivalent and, moreover, the sets 
\[
    \Omega_E(F, \varepsilon) = \set{\xi \in \Omega_E(G): \text{$\xi$ is contained in the same connected component of $G\setminus F$ as $\varepsilon$}}
\]
are open in $\Omega_E(G)$. 

This naturally raises a few questions: if we choose to define an end space using finite sets of timid vertices instead of finite sets of edges, what relation can we expect between such space and the edge-end space? And what can be said about directions which are defined similarly? We seek to answer these questions in this section.

\subsection{Timid-end spaces as edge-end spaces}\label{SecTimid-end1}

\paragraph{}
It is straightforward to define the timid-end space: let $\Omega_{\mathrm{t}}(G)$ denote the quotient of the set of $G$-rays by the $\mathrm
t$-equivalence in which two rays are identified if no finite set of timid vertices can separate tails of said two rays. The elements of $\Omega_{\mathrm{t}}(G)$ shall be called \emph{timid-ends} and the basic open sets $\Omega_{\mathrm t}(F,\varepsilon)$ in $\Omega_{\mathrm{t}}(G)$ are defined just as in $\Omega(G)$, but using finite subsets $F\subset \mathrm{t}(G)$ and timid-ends.

It is not hard to show that, given a graph $G$, its timid-end and its edge-end spaces can be different. For instance:

\begin{example}
    Let $K$ and $K'$ be two disjoint cliques of size $\aleph_0$, with fixed $v\in K$ and $v'\in K'$. Let $G$ be the graph such that 
    \begin{align*}
        \mathrm{V}(G) &= \mathrm{V}(K)\sqcup\mathrm{V}(K')\\
        \{u,w\}\in \mathrm{E}(G) &\iff \left(\{u,w\}\in \mathrm{E}(K)\sqcup\mathrm{E}(K')\right) \;\text{or}\; \left(u=v \text{ and } w=v'\right).
    \end{align*}
    
    Then $\Omega_E(G)$ is clearly the two-points discrete space. On the other hand, $G$ contains no timid vertex, so $\Omega_\mathrm{t}(G)$ is the trivial single-point space.
\end{example}

Nevertheless, we will show that the class of timid-end spaces is still the same as the class of edge-end spaces. In order to do so, let $\mathrm{E}_\mathrm{t}(G)\subset \mathrm{E}(G)$ be such that $e\in \mathrm{E}_\mathrm{t}$ if, and only if, some timid vertex is incident in $e$ (we simply write $\mathrm{E}_\mathrm{t}$ when there is no risk of confusion about which $G$ is being considered).

Consider the graph denoted by $G/[v_0]_E$, in which 
\begin{align*}
 \mathrm{V}(G/[v_0]_E&) = \{[v_0]_E\}\cup \mathrm{V}(G)\setminus [v_0]_E, \\
 \{w,u\}\in \mathrm{E}(G/[v_0]_E) &\iff \{w,u\}\in \mathrm{E}(G)\\
 \{[v_0]_E,u\}\in \mathrm{E}(G/[v_0]_E) &\iff \exists v\in[v_0]_E \left(\{v,u\}\in \mathrm{E}(G)\right)
\end{align*}
for all $u,w\in \mathrm{V}(G)\setminus [v_0]_E$ (where $[v_0]_E$ is the $\sim_E$-equivalence class of $v_0$, as described in \myref{EQ_VertexEdgeEqv}). 

Now let $\pi\colon \mathrm{V}(G) \to \mathrm{V}(G/[v_0]_E)$ be the quotient projection. Then it is easy to see that, given $e\in \mathrm{E}(G/[v_0]_E)$, the set
\[\pi^{-1}e \doteq \set{\{x,y\}\in \mathrm{E}(G): e = \{\pi(x),\pi(y)\}}\]
is finite. 

Consider the following generalization of the notion presented in \myref{EQ_VertexEdgeEqv}: given $u,v\in \mathrm{V}(G)$ and $A\subset \mathrm{E}(G)$, we write $u\sim_A v$ if, for every finite $F\subset A$, $u$ and $v$ share the same connected component in $G\setminus F$. In this case:

\begin{lemma}\label{LEMMA_TimidEqvClassQuot}
    Suppose $u\nsim_E v$. Then $u\sim_{E_\mathrm{t}} v$ in $G$ if, and only if, $[u]_E\sim_{E_\mathrm{t}} v$ in $G/[u]_E$.
\end{lemma}
\begin{proof}
    Suppose $u\nsim_{E_\mathrm{t}} v$ in $G$. Then there is a finite $F\subset E_\mathrm{t}(G)$ separating $u$ from $v$. Let
    \[\pi[F] \doteq \set{\{\pi(x),\pi(y)\}:\{x,y\}\in F}.\]

    Then, since $w\in \mathrm{V}(G)$ is timid in $G$ if, and only if $\pi(w)$ is timid in $G/[u]_E$, $\pi[F]\subset E_\mathrm{t}(G/[u]_E)$. 
    
    We claim that $\pi[F]$ separates $[u]_E$ from $v$: indeed, suppose that $P=\seqq{[u]_E, v_1, \dotsc, v}$ is a $(G/[u]_E)$-path. Fix $u'\in [u]_E$ such that $u'\sim_E u$ and $\{u',v_1\}\in \mathrm{E}(G)\setminus F$. Then there is a $G$-path $P'$ from $u$ to $u'$ which avoids $F$. Since there can be no $G$-path from $u$ to $v$ which avoids $F$, it follows that $\seqq{v_1, \dotsc, v}$ does not avoid $F$, in which case $P$ does not avoid $\pi[F]$, as desired.

    Now suppose $[u]_E\nsim_{E_\mathrm{t}} v$ in $G/[u]_E$. Then there is a finite $F\subset E_\mathrm{t}(G/[u]_E)$ which separates $[u]_E$ from $v$ in $G/[u]_E$. Consider 
    \[\pi^{-1}(F) \doteq \bigcup_{e\in F}\pi^{-1}e.\]
    Then, from what has already been discussed, $\pi^{-1}(F)$ is finite and contained in $E_\mathrm{t}(G)$.

    We claim that $\pi^{-1}(F)$ separates $u$ from $v$: indeed, suppose $P$ is a $G$-path from $u$ to $v$. Then the projection by $\pi$ of $P$ is a walk from $[u]_E$ to $v$. Since $F$ separates $[u]_E$ from $v$, the projection of $P$ must not avoid $F$. Thus, if $e\in F$ is also in the projection of $P$, then $\pi^{-1}e\cap \mathrm{E}(P)\neq \emptyset$, which concludes the proof.   
\end{proof}

\begin{lemma}\label{LEMMA_TimidEdgeSeparation}
    Suppose $u,v\in \mathrm{V}(G)$ are such that $u$ is timid and $u\nsim_E v$. Then $u\nsim_{E_\mathrm{t}} v$.
\end{lemma}
\begin{proof}
    Without loss of generality, in view of Lemma \ref{LEMMA_TimidEqvClassQuot}, we may assume that $u\nsim_E u'$ for every $u'\in \mathrm{V}(G)$.

    Striving for a contradiction, suppose that $u\sim_{E_\mathrm{t}}v$. We recursively construct a sequence $\seq{(P_n,F_n):{n\in\omega}}$ such that, for every $n\in\omega$, 
    \begin{itemize}
        \item[(i)] $F_n\subset \mathrm{E}(G)$ is finite;
        \item[(ii)] $\mathcal P_{n+1}$ is a finite set of paths from $u$ to $\mathrm{V}(P_n)$ with edges in $F_n$;
        \item[(iii)] $F_{n}$ separates $u$ from $\mathrm{V}(\mathcal P_n)$; 
        \item[(iv)] $u\sim_{E_\mathrm{t}} w$ in $G\setminus \left(\bigcup_{i< n}F_i\right)$ for some $w\in \mathrm{V}(\mathcal P_n)$. 
    \end{itemize}
    
    For starters, consider a maximal collection $\mathcal P_0$ of pairwise edge-disjoint paths from $u$ to $v$. Since $u\nsim_E v$, $\mathcal P_0$ is finite. 
    
    We claim that there is a $w\in \mathrm{V}(\mathcal P_0)$ such that $u\sim_{E_\mathrm{t}} w$ in $G\setminus \mathrm{E}(\mathcal P_0)$. Suppose not, i.e., that there is a finite $F\subset E_\mathrm{t}$ separating $u$ from $\mathrm{V}(\mathcal P_0)$ in $G\setminus \mathrm{E}(\mathcal P_0)$. Then such $F$ cannot separate $u$ from $\mathrm{V}(\mathcal P_0)$ in $G$, otherwise said $F$ would also separate $u$ from $v$ in $G$, contradicting our hypothesis. Thus, there must be some path $P$ in $G$ connecting $u$ to $\mathrm{V}(\mathcal P_0)$ which does not go through any edge in $F$. Let $w'$ be the first vertex from $\mathrm{V}(\mathcal P_0)$ which appears in $P$. Then the truncation of $P$ from $u$ to $w'$ does not go through any edge in $\mathrm{E}(\mathcal P_0)$, thus $F$ should break $P$, a contradiction. 
    
    Now, If $w\in \mathrm{V}(\mathcal P_0)$ is such that $u\nsim_{E_\mathrm{t}}w$, fix a finite $F_w\subset E_\mathrm{t}$ which separates $u$ from $w$ and then let
    \[
      F_0 = \mathrm{E}(\mathcal P_0)\cup \bigcup_{\substack{w\in \mathrm{V}(\mathcal P_0)\\w\nsim_{E_\mathrm{t}}u}}F_w.
    \]

    In this case, because $\bigcup_{\substack{w\in \mathrm{V}(\mathcal P_0)\\w\nsim_{E_\mathrm{t}}u}}F_w\subset E_\mathrm{t}$ is finite, there is still a $w\in \mathrm{V}(\mathcal P_0)$ such that $u\sim_{E_\mathrm{t}} w$ in $G\setminus F_0$. Thus, it is clear that (i)--(iv) hold so far.

    Suppose $(\mathcal P_i,F_i)_{i\le n}$  are defined in a way which satisfies (i)--(iv) thus far. Consider a maximal collection $\mathcal P_{n+1}$ of pairwise edge-disjoint paths from $u$ to $\mathrm{V}(\mathcal P_n)$ in $G\setminus \left(\bigcup_{i\le n}F_i\right)$. Since $u\nsim_E w$ for any $w\in \mathrm{V}(\mathcal P_n)$, $\mathcal P_{n+1}$ is finite.

    We claim that there is a $w\in \mathrm{V}(\mathcal P_n)$ such that $u\sim_{E_\mathrm{t}} w$ in $G\setminus \left(\mathrm{E}(\mathcal P_{n+1})\cup\bigcup_{i\le n}F_i\right)$. Suppose not, i.e., that there is a finite $F\subset E_\mathrm{t}$ separating $u$ from $\mathrm{V}(\mathcal P_{n+1})$ in $G\setminus \left(\mathrm{E}(\mathcal P_{n+1})\cup\bigcup_{i\le n}F_i\right)$. Then such $F$ cannot separate $u$ from $\mathrm{V}(\mathcal P_{n+1})$ in $G$, otherwise said $F$ would also separate $u$ from $v$ in $G$, contradicting our hypothesis. Thus, there must be some path $P$ in $G$ connecting $u$ to $\mathrm{V}(\mathcal P_{n+1})$ which does not go through any edge in $F$. Let $w'$ be the first vertex from $\mathrm{V}(\mathcal P_{n+1})$ which appears in $P$. Then the restriction of $P$ from $u$ to $w'$ does not go through any edge in $\mathrm{E}(\mathcal P_{n+1})$, thus $F$ should break $P$, a contradiction. 
    
    Now, If $w\in \mathrm{V}(\mathcal P_{n+1})$ is such that $u\nsim_{E_\mathrm{t}}w$, fix a finite $F_w\subset E_\mathrm{t}$ which separates $u$ from $w$ and then let
    \[
      F_{n+1} = \mathrm{E}(\mathcal P_{n+1})\cup \bigcup_{\substack{w\in \mathrm{V}(\mathcal P_{n+1})\\w\nsim_{E_\mathrm{t}}u}}F_w.
    \]

    In this case, because $\bigcup_{\substack{w\in \mathrm{V}(\mathcal P_{n+1})\\w\nsim_{E_\mathrm{t}}u}}F_w\subset E_\mathrm{t}$ is finite, there is still a $w\in \mathrm{V}(\mathcal P_{n+1})$ such that $u\sim_{E_\mathrm{t}} w$ in $G\setminus \left(\cup\bigcup_{i\le n+1}F_i\right)$. Thus, it is clear that (i)--(iv) hold so far and our recursion is done. Now consider the graph 
    \[
        H = \bigcup_{n\in\omega}\left(\bigcup\mathcal P_n\setminus \{u\}\right).
    \]
    
    We claim that $H$ is connected. Indeed, note that $v\in \mathrm{V}(H)$. We inductively show that if $w\in \mathrm{V}(\mathcal P_n)$ for some $n\in\omega$, then there is a concatenation of segments of paths in $\bigcup_{k\le n}\mathcal{P}_k$ connecting $w$ to $v$. Suppose that this is true for every $w\in \bigcup_{k\le n} \mathrm{V}(\mathcal P_k)$ and let $w\in \mathcal P_{n+1}$ be given. By construction of $\mathcal{P}_{n+1}$, $w\in \mathrm{V}(P)$ for some path $P$ from $u$ to some $w'\in \mathrm{V}(\mathcal P_n)$. Thus, concatenating the segment of $P$ starting at $w$ with the path $P'$ given by our induction hypothesis to $w'$ we reach the desired conclusion. 
    
    Moreover, $H$ is locally finite: indeed, it follows from (ii) and (iii) that $F_n$ separates $\mathrm{V}(\mathcal P_n)$ from $\mathrm{V}(\mathcal P_{n+1})$ in $H$. The Star-Comb Lemma shows us that $H$ contains some ray $r$ (since $H$ is infinite, by construction) which passes through infinitely many of the $\mathcal P_n$'s. Note that if $w\in \mathrm{V}(\mathcal P_n)\cap r$ and $w'\in \mathrm{V}(\mathcal P_m)\cap r$ for $n\neq m$, then there are $H$-paths $P\in \mathcal P_n$ and $P'\in \mathcal P_m$ passing through $w$ and $w'$, respectively. But, by construction, $\mathrm{E}(\mathcal P_n)\cap \mathrm{E}(\mathcal P_m)=\emptyset$, so it follows that $P$ and $P'$ are edge-disjoint. We at last conclude that $u$ dominates $r$, contradicting our hypothesis that $u$ is timid. 
\end{proof}

\begin{theorem}\label{THM_EdgeEndsHomeoTimidEnds}
    The class $\Omega_{\mathrm{t}}$ of timid-end spaces is the same as the class $\Omega_E$ of edge-end spaces. 
\end{theorem}
\begin{proof}
    The inclusion $\Omega_E\subset \Omega_{\mathrm{t}}$ is easily realized by noting that, given a graph $G$, letting $H$ be the graph obtained by taking a single-vertex subdivision of every edge in $G$, $\Omega_{\mathrm{t}}(H) \approx \Omega_E(G)$.

    Now, for the remaining inclusion $\Omega_{\mathrm{t}}\subset \Omega_E$, let $G$ be a graph. For each edge $e = \{v_0,v_1\}\in \mathrm{E}(G)\setminus \mathrm{E}_\mathrm{t}(G)$, let $U_e = \set{u_n^e:n\in \omega}$ be an infinite set and  
    \[
    E_e = \set{\{u_n^e, v_i\}:n\in\omega, i\in \{0,1\}}.
    \]
    Then let $H$ be the graph, locally illustrated in \myref{FIGURE_Timid-to-Edges}, such that 
    \begin{align*}
        \mathrm{V}(H) &= \mathrm{V}(G)\cup \bigcup_{e\in E(G)\setminus E_\mathrm{t}(G)}U_e,\\
        \mathrm{E}(H) &= \mathrm{E}(G)\cup \bigcup_{e\in \mathrm{E}(G)\setminus \mathrm{E}_\mathrm{t}(G)}E_e.
    \end{align*}
    
    \begin{figure}[ht]
        \centering
        \begin{tikzpicture}
        \def\radius{0.15}
        \def\dy{2}
        \def\dx{0.75}
    
        \begin{scope}
            \node (v0G) at (2.5*\dx-0.25,\dy/2-0.15) {\small$v_0$};
            \node (v1G) at (2.5*\dx-0.25,-\dy/2+0.15) {\small$v_1$};
            
            \draw[->] (-0.5*\dx,\dy)-- (5.5*\dx,\dy);
            \draw[->] (-0.5*\dx,-\dy)-- (5.5*\dx,-\dy);
    
            \foreach \i in {0,...,5}
            {
                \draw (2.5*\dx,\dy/2) -- ({\i*\dx},\dy);
                \draw (2.5*\dx,-\dy/2) -- ({\i*\dx},-\dy);
            }

            \draw (2.5*\dx,\dy/2) -- (2.5*\dx,-\dy/2);
            \node (eG) at (2.5*\dx-0.15,0) {\small$e$};

            \draw[thick] (5.5*\dx+0.75,\dy+0.75)  rectangle (-0.5*\dx-0.25,-\dy-0.25);
            \node (G) at (5.5*\dx+0.5,\dy+0.5) {$G$};
        \end{scope}

        \draw[->, thick] (6-0.5,0) -- (6+0.5,0);

        \begin{scope}[shift = {(10*\dx,0)}]
            \node (v0H) at (2.5*\dx-0.25,\dy/2-0.15) {\small$v_0$};
            \node (v1H) at (2.5*\dx-0.25,-\dy/2+0.15) {\small$v_1$};
            
            \draw[->] (-0.5*\dx,\dy)-- (5.5*\dx,\dy);
            \draw[->] (-0.5*\dx,-\dy)-- (5.5*\dx,-\dy);
    
            \foreach \i in {0,...,5}
            {
                \draw (2.5*\dx,\dy/2) -- ({\i*\dx},\dy);
                \draw (2.5*\dx,-\dy/2) -- ({\i*\dx},-\dy);
            }

            \draw (2.5*\dx,\dy/2) -- (2.5*\dx,-\dy/2);
            \node (eH) at (2.5*\dx-0.15,0) {\small$e$};

            \draw[thick] (5.5*\dx+0.75,\dy+0.75)  rectangle (-0.5*\dx-0.25,-\dy-0.25);
            \node (G) at (5.5*\dx+0.5,\dy+0.5) {$H$};

            \foreach \i in {1,...,2}
            {
                \draw (2.5*\dx,\dy/2) -- (2.5*\dx+0.75*\dx*\i,0);
                \draw (2.5*\dx,-\dy/2) --  (2.5*\dx+0.75*\dx*\i,0);
            }
            \node (u0) at (2.5*\dx+0.75*\dx+0.15,0) {\tiny$u_0^e$};
            \node (u1) at (2.5*\dx+0.75*\dx*2+0.15,0) {\tiny$u_1^e$};
            \node (udots) at (2.5*\dx+0.75*\dx*3+0.15,0) {\tiny $\cdots$};
        \end{scope}
    \end{tikzpicture}
    \caption{Adjacent edge-dominant vertices in $G$ cannot be separated by finite edges in $H$.}
\label{FIGURE_Timid-to-Edges}
    \end{figure}
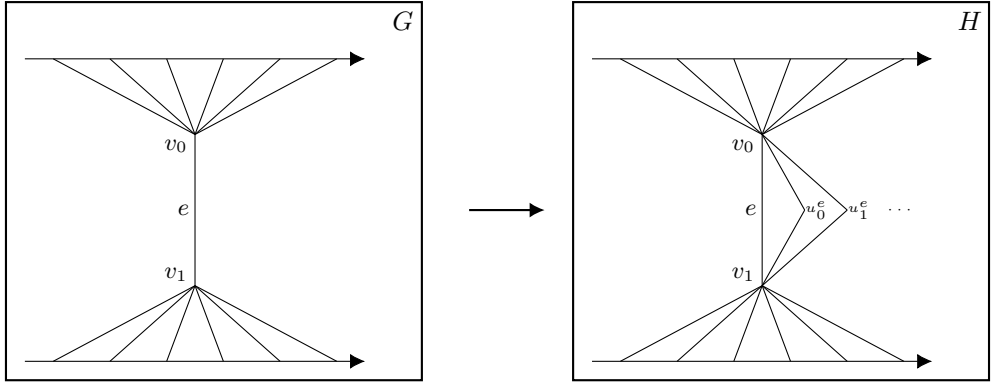
    
    Without any risk of confusion, we identify $G$ as an (induced) sub-graph of $H$.

    We claim that $\psi\colon \Omega_{\mathrm{t}}(G)\to \Omega_E(H)$ such that $\psi([r]_{\mathrm{t}(G)}) = [r]_{E(H)}$ is a homeomorphism. 
    \begin{description}
        \item[The association $\psi$ is a well defined map:] Suppose $r\sim_{T}r'$ in $G$ and a finite $F\subset \mathrm{E}(H)$ is given. Note that $F\setminus \mathrm{E}_\mathrm{t}(G)$ does not separate any component of $H$, so it suffices to assume that $F\subset \mathrm{E}_\mathrm{t}(G)$. Thus, if we let \[F' = \set{u\in \mathrm{t}(G): \text{ $u$ is incident in some $e\in F$}},\] we obtain a $G$-path $P$ between the $(G\setminus F')$-tails of $r$ and $r'$ which avoids $F'$ and, hence, $F$ as well.

        \item[The map $\psi$ is surjective:] Let $r$ be a ray in $H$. Then it is easy to see that there is a ray $r'$ in $G$ which passes through infinitely many vertices of $r$, thus $\psi(r') = [r']_{\mathrm{E}(H)} = [r]_{\mathrm{E}(H)}$.

        \item[The map $\psi$ is injective:] Let $r$ and $r'$ be rays in $G$ such that $r\nsim_{\mathrm{t}(G)} r'$ in $G$. Then there is a finite $F\subset \mathrm{t}(G)$ separating the $(G\setminus F)$-tails of $r$ and $r'$. Since each $u\in F$ is timid, there is a set $F_u\subset \mathrm{E}(G)$ separating $u$ from the tails of both $r$ and $r'$ in $G\setminus F_u$. Now, if $e=\{x,y\}\in F_u\setminus \mathrm{E}_\mathrm{t}(G)$, then $x$ and $y$ are dominant vertices, which in turn tells us that $u\nsim_{\mathrm{E}(G)} x$ and $u\nsim_{\mathrm{E}(G)} y$. Hence, by \myref{LEMMA_TimidEdgeSeparation}, we can find $F_e\subset \mathrm{E}_\mathrm{t}(G)$ separating $u$ from both $x$ and $y$. 
        
        We claim that 
        \[F' = \bigcup_{u\in F}\left((F_u\cap \mathrm{E}_\mathrm{t}(G))\cup \left(\bigcup_{e\in F_u\setminus \mathrm{E}_\mathrm{t}(G)}F_e\right)\right)\subset \mathrm{E}_\mathrm{t}(G)\]
        separates the $(H\setminus F)'$-tails of $r$ and $r'$. Indeed, let $P$ be an $H$-path between such tails. Then it is easy to see that $P$ must pass through some $u\in F$. But $F'$ separates $F$ from the $(G\setminus F)$-tails of both $r$ and $r'$, so $P$ cannot avoid $F'$. It follows that $\psi([r]_{T}) = [r]_{\mathrm{E}(H)} \neq [r']_{\mathrm{E}(H)} = \psi([r']_{T})$, so $\psi$ is injective.

        \item[The map $\psi$ is continuous:] Suppose a finite $F\subset \mathrm{E}(H)$ is given. Again, because $F\setminus E_\mathrm{t}(G)$ does not separate any component of $H$, it suffices to assume that $F\subset \mathrm{E}_\mathrm{t}(G)$. Let \[F' = \set{u\in \mathrm{t}(G): \text{ u is incident in some $e\in F$}},\] and suppose a ray $r$ in $G$ is given. Let $P$ be a $(G\setminus F')$-path between the tails of $r$ and some ray $r'$ in $G\setminus F'$. Since $P$ avoids $F'$, it also avoids $F$, so the $(G\setminus F)$-tails of $r'$ and $r$ are in the same connected component of $G\setminus F$.

        \item[The application $\psi$ is open:] Suppose a finite $F\subset \mathrm{t}(G)$ is given with $\varepsilon\in \Omega_{\mathrm{t}(G)}(G)$ and let $\xi\in \Omega_{E}(H)$ be such that $\xi\in \psi[\Omega_{\mathrm{t}(G)}(F,[r]_{\mathrm{t(G)}})]$. In view of $\psi$'s surjectivity, we may pick a ray $r$ inside the same component as $\varepsilon$ in $G\setminus F$ such that $[r]_{\mathrm{E}(H)} = \xi$. Since each $u\in F$ is timid in $G$, there is a set $F_u\subset \mathrm{E}(G)$ separating $u$ from $r$. Now, if $e=\{x,y\}\in F_u\setminus \mathrm{E}_\mathrm{t}(G)$, then $x$ and $y$ are edge-dominant vertices, which in turn tells us that $u\nsim_{\mathrm{E}(G)} x$ and $u\nsim_{E(G)} y$. Hence, by \myref{LEMMA_TimidEdgeSeparation}, we can find $F_e\subset \mathrm{E}_\mathrm{t}(G)$ separating $u$ from both $x$ and $y$ in $G$. 
        
        We claim that
        \[F' = \bigcup_{u\in F}\left((F_u\cap \mathrm{E}_\mathrm{t}(G))\cup \left(\bigcup_{e\in F_u\setminus \mathrm{E}_\mathrm{t}(G)}F_e\right)\right)\subset \mathrm{E}_\mathrm{t}(G)\]
        separates the $(H\setminus F')$-tail of $r$ from $F$ in $H$: indeed, suppose that $P$ is an $H$-path from the $(G\setminus F')$-tail of $r$ to $F$. By construction of $H$, there is a $G$-path $P'$ such that $\mathrm{E}(P)\cap \mathrm{E}_{\mathrm{t}(G)} = \mathrm{E}(P')\cap \mathrm{E}_{\mathrm{t}(G)}$. Since $F'$ separates $F$ from $(G\setminus F')$-tail of $r$ in $G$, $P'$ cannot avoid $F'$. But $F'\subset \mathrm{E}_{\mathrm{t}(G)}$, so $P$ cannot avoid $F'$ as well.

        Now suppose that $P$ is an $(H\setminus F')$-path from $v\in \mathrm{V}(H)$ to the $H\setminus F'$-tail of $r$. Then, since $F'$ separates the $(H\setminus F')$-tail of $r$ from $F$ in $H$, $P$ must avoid $F$ as well. Hence, the connected component in $H\setminus F'$ of the $(H\setminus F')$-tail of $r$ is contained in the connected component of $\psi(\varepsilon)$ in $H\setminus F$ and
        the proof is complete.
    \end{description}
\end{proof}

\subsection{Timid-direction spaces as timid-end spaces}\label{timid-direciton-end2}

\paragraph{} We now want to define timid-direction spaces and show that they are homeomorphic to their respective timid-end spaces -- but we will present such proof as a consequence of a result in a more general setting. 

For any $U \subset \mathrm{V}(G)$, it is also straight forward to define the following topological spaces:
\begin{itemize}
    \item[(i)] The \emph{$U$-direction space} $\mathcal{D}_U(G)$ where we only change $[\mathrm{V}(G)]^{< \aleph_0}$ in the definition of $\mathcal{D}(G)$ for $[U]^{< \aleph_0}$. The basic open sets $\mathcal{D}_U(F,\rho)$ here are defined just as in $\mathcal{D}(G)$ as well, but using finite subsets $F\subset U$ and $U$-directions.
    \item[(ii)] The \emph{$U$-end space} $\Omega_U(G)$ is the quotient of the $G$-rays by the $U$-ray equivalence, meaning, two rays are \emph{identified} when no finite subset of $U$ can separate these rays. The points of the space are the $U$-classes $[r]_U \in \Omega_U(G)$. Again, the basic open sets $\Omega_U(F,\varepsilon)$ here are defined just as in $\Omega(G)$, but using finite subsets $F\subset U$ and $U$-ends.
\end{itemize}

Given a ray $r$ in $G$, let $\rho_r\in \mathcal{D}_U(G)$ be such that $\rho_r(F)$ contains a tail of $r$ in $G\setminus F$ for every finite $F\subset U$. Note that if $r\sim_U r'$, then $\rho_r = \rho_{r'}$. Hence, the map $\rho_\bullet\colon \Omega_U(G)\to \mathcal{D}_U(G)$ such that $\varepsilon = [r]_U \mapsto \rho_r$ is well defined and we may write $\rho_\varepsilon$ for $\rho_r$.

Note that $\Omega_{\mathrm{t}}(G)= \Omega_{\mathrm{t}(G)}(G)$, so we let $\mathcal{D}_{\mathrm{t}}(G)\doteq \mathcal{D}_{\mathrm{t}(G)}(G)$. We now strive to characterize for which cases of $U\subset \mathrm{V}(G)$ the map $\rho_\bullet$ is a homeomorphism.

It should be clear that $\rho_\bullet$ is injective, continuous and open over its image (since $\rho_r\in \mathcal{D}_U(F,\rho_\varepsilon)$ if, and only if $[r]_U\in \Omega_U(F,\varepsilon)$ for every finite $F\subset U$). Thus:

 \begin{proposition}\label{PROP_UTopEmbed}
       The map $\rho_\bullet\colon \Omega_U(G)\to \mathcal{D}_U(G)$ is a topological embedding for any $U\subset \mathrm{V}(G)$.
 \end{proposition}

It follows that $\rho_\bullet$ is a homeomorphism if, and only if, $\rho_\bullet$ is surjective. This is not the case for every $U\subset \mathrm{V}(G)$, though: 

\begin{example}\label{EX_rhobullNotSurj}
    Suppose $G$ is a star of rays with center vertex $c\in \mathrm
    V(G)$ and let $U = \mathrm V(G)\setminus \{c\}$. Let $\rho_c$ denote the $U$-direction such that $\rho_c(F)$ is the connected component containing $c$ (note that it is always infinite). It is easy to see that each $G$-ray can have its tail separated from $c$ by some $F\subset U$, so $\rho_c$ cannot be equal to $\rho_r$ for any $G$-ray $r$.
\end{example}

So let us determine exactly for which cases of $U\subset \mathrm{V}(G)$ the map $\rho_\bullet$ is surjective: in what follows, we say that a vertex $v\in \mathrm{V}(G)$ is $U$-timid in $G$ if for every $G$-ray $r$ there exists a finite $F\subset U\setminus \{v\}$  separating $v$ from the $(G\setminus F)$-tail of $r$. Furthermore, we will say that $v$ is $U$-dense if for every finite $F\subset U\setminus\{v\}$, $v$ is contained on an infinite connected component of $G\setminus F$. We will denote by $\partial \mathrm{t}_U$ the set of all vertices in $G$ which are both $U$-timid and $U$-dense. 

Note that the chosen set $U$ in \myref{EX_rhobullNotSurj} is such that $\partial\mathrm t_U = \mathrm
V(G) \not\subset U$. As the following theorem tells us, this is no coincidence.

\begin{theorem}\label{THM_piDSurj}
    Let $U\subset \mathrm{V}(G)$ be given. Then for every $\rho\in \mathcal{D}_U(G)$ there exists a ray $r$ in $G$ such that $\rho_r = \rho$ if, and only if, $\partial \mathrm{t}_U \subset U$.
\end{theorem}
\begin{proof}
    Suppose that there is some $v\in \partial \mathrm{t}_U\setminus U$. Then, since $v$ is $U$-dense, let $\rho_v$ be the direction which selects the connected component of $v$ in $G\setminus F$ for any finite $F\subset U$. Since $v$ is $U$-timid, for every $G$-ray $r$ there is some finite $F\subset U$ separating $v$ from the $(G\setminus F)$-tail of $r$, so there can be no ray $r$ in $G$ such that $\rho_r = \rho_v$.

    Now suppose that $\partial \mathrm{t}_U \subset U$ (in what follows we adapt the proof of Theorem 2.2 from \cite{DIESTEL2003197}). Let $\rho\in \mathcal{D}_U(G)$ be given. For each $F\subset U$, let us write $\hat{F} = \mathrm{V}(\rho(F))\cup \mathrm{N}(F,\rho(F))$ for the set of vertices in $\rho(F)$ and their neighbors in $F$. Thus, $\hat{F}\supset\hat{F'}$ whenever $F\subset F'$. Let
    \[
        U^\rho = \bigcap_{F\in [U]^{<\aleph_0}}\hat{F}.
    \]

    Note that if $v\in U^\rho$, then $v$ is $U$-dense: indeed, if there is some finite $F\subset U\setminus \{v\}$ such that the connected component of $G\setminus F$ containing $v$ is finite, then $\rho(F)$ must not contain $v$, so $v\notin U^\rho$. Hence, if $v\in U^\rho\setminus U$, since $\partial \mathrm{t}_U\subset U$, $v$ must not be $U$-timid. In this case, any $G$-ray $r$ which cannot be separated by any finite $F\subset U$ from $v$ will clearly be such that $\rho_r = \rho$. Thus, without loss of generality, we may assume that $U^\rho\subset U$.
    
    Suppose $U^\rho$ is infinite and fix $v_0\in U^\rho$. Since $U^\rho\subset \rho(\emptyset)$, which is a connected component of $G$, we may take $v_1\in U^\rho\setminus \{v_0\}$ as a vertex with a path $P_0 = \seqq{v_0, \dotsc, v_1}$ in $G$ such that $(\mathrm{V}(P_0)\setminus\{v_0,v_1\})\cap U^\rho = \emptyset$. For each $w\in \mathrm{V}(P_0)\setminus\{v_0,v_1\}$, since $w\notin U^\rho$, there must be a finite $F\subset U$ such that $w$ is not in $\rho(F_w)$. Let 
    \[F_1 = \{v_0,v_1\}\cup \bigcup_{w\in \mathrm{V}(P_0)\setminus\{v_0,v_1\}}F_w.\]
    
    The definition of $U^\rho$ implies that $v_1$ has some neighbor $u\in \rho(F_1)$. If $u\in U^\rho$, let $v_2 = u$. Otherwise, let $v_2\in \rho(F_1)\cap U^\rho$ be such that there is some path $P_1 = \seqq{u, \dotsc, v_2}$ in $\rho(F_0)$ with no intermediate vertex in $U^\rho$. By proceeding in this manner we construct a ray $r$ which passes through infinitely many vertices of $U^\rho$, so that $\rho_r = \rho$.

    At last, suppose that $U^\rho\subset U$ is finite. If $\rho(U^\rho)\cap U=\emptyset$, then any vertex $v\in \rho(U^\rho)$ is $U$-dense and it thus follows from the fact that $\partial \mathrm{t}_U\subset U$ that $v$ cannot be $U$-timid. In this case, any ray $r$ which cannot be separated from $v\in \rho(U^\rho)$ by a finite subset of $U$ will suffice to show that $\rho_{[r]}^U = \rho$. Hence, without loss of generality, we may assume that $U^\rho=\emptyset$, that $G$ is connected and that $U\neq\emptyset$ (by considering $\rho(U^\rho)$ in place of $G$). 
    
    We define an infinite sequence $\seq{F_n:n\in \mathbb{N}}$ of pairwise disjoint finite subsets of $U$ such that $\rho(F_n)$ contains both $F_{n+1}$ and $\rho(F_{n+1})$ (so that $F_{n+1}$ separates $F_n$ from $\rho(F_{n+1})$) as follows: first, fix any $F_0 = \{v_0\}\subset U$. Now suppose $\seq{F_i:i\le n}$ is already defined. Then for every $v\in F_n$, since $v\notin U^\rho$, there must be some $F_v\subset U\cap \rho(F_n)$ such that $v\notin \hat{F_v}$. In this case, let $F = \bigcup_{v\in F_n}F_v$ and
    \[F_{n+1} = \set{u\in F: \text{ $u$ neighbors someone in $F_n$}}\]
    and note that $\rho(F_{n+1}) = \rho(F)$ (since $F_{n+1}\subset F$ implies that $\rho(F)\subset \rho(F_{n+1})$ and $\rho(F)$ is already a connected component of $G\setminus F_{n+1}$). Obviously, $F_{n+1}\subset \rho(F_n)$ and, since $\rho(F)\subset G\setminus F_n$ by construction, $\rho(F_{n+1}) = \rho(F)\subset \rho(F_n)$ and our recursion is complete.

    Now let $v\in \rho(F_{n+1})$. Then every path from $v$ to $F_0$ must pass through $F_i$ for every $0<i\le n$ Thus, since $\seq{F_i:i\le n}$ are pairwise disjoint, the distance from $v$ to $F_0$ is at least $n$ and we can deduce that 
    \[
    \bigcap_{n\in\omega}\hat{F}_n = \emptyset.
    \]
    
    Pick a $v_n\in \rho(F_n)$ for each $n\in\omega$. We claim that there must be a comb with infinitely many teeth in $W = \set{v_n:n\in\omega}$: indeed, we infer from \myref{starcomb} that there exists either a star or a comb with infinitely many tips in $W$ -- however, there can be no such star, for its center would otherwise be in $\bigcap_{n\in\omega}\hat{F}_n$. 
    
    Let $r = \seq{u_n:n\in\omega}$ be the spine of said comb with infinitely many pairwise vertex-disjoint paths $\set{P_n=(u_{k_n}, \dotsc, v_{m_n}):n\in\omega}$ from $r$ to $W$ and suppose a finite $F\subset U$ is given. Choose $n\in\omega$ large enough so that $F\cap \rho(F_n)=F\cap \mathrm{V}(P_n)=\emptyset$ and $u_i\notin F$ for every $i\ge k_n$. Then $\rho(F_n)$ is a connected subgraph of $G\setminus F$, so the connected component of $G\setminus F$ containing it must be $\rho(F)$ (since $\emptyset\neq\rho(F\cup F_n)\subset \rho(F)\cap\rho(F_n)$). Thus, since $F\cap \mathrm{V}(P_n)=\emptyset$, $\mathrm{V}(P_n)\subset \rho(F)$. But $u_i\notin F$ for every $i\ge k_n$, so $\set{u_i:i\ge k_n}\subset \rho(F)$. This shows that $\rho_r^U = \rho$, which concludes the proof.
\end{proof}

\begin{corollary}\label{COR_rhoUHomeo}
    Provided that $\partial \mathrm{t}_U \subset U$, the map $\rho_\bullet \colon \Omega_U(G)\to \mathcal{D}_U(T)$ is a homeomorphism.
\end{corollary}

\begin{corollary}[Theorem 2.2 in \cite{DIESTEL2003197}]\label{THM_DiestelSurj}
    For every graph $G$, the map $\rho_\bullet\colon \Omega(G)\to \mathcal{D}(G)$ is a homeomorphism.
\end{corollary}
\begin{proof}
    Apply \myref{COR_rhoUHomeo} to the case $U=\mathrm{V}(G)$. 
\end{proof}

And, as promised:

\begin{corollary}\label{COR_TimidDirectionsHomeoTimidEnds}
    For every graph $G$, the map $\rho_\bullet\colon \Omega_{\mathrm{t}}(G)\to \mathcal{D}_{\mathrm{t}}(G)$ is a homeomorphism.
\end{corollary}
\begin{proof}
     Note that if $v$ is $U$-timid for some $U\subset \mathrm{V}(G)$, then $v$ is, in particular, timid. Hence, $\partial \mathrm{t}_U \subset \mathrm{t}(G)$ for any $U\subset \mathrm{V}(G)$ and the result follows from applying \myref{COR_rhoUHomeo} to $U=\mathrm{t}(G)$.
\end{proof}

\begin{rmk}\label{RMK_NotQuotient}
    Note that, mapping the class of a ray in $\Omega(G)$ to the $U$-class of the same ray gives rise to a canonical $\varepsilon \doteq [r] \in \Omega(G) \mapsto \pi_\Omega(\varepsilon) \doteq [r]_U \in \Omega_U(G)$. Analogously, $\rho \in \mathcal{D}(G) \mapsto \pi_\mathcal{D}(\rho)  = \rho^U\doteq \rho \upharpoonright [U]^{< \aleph_0} \in \mathcal{D}_U(G)$ is a well defined map. In this case, the diagram
\[
    \begin{CD}
    \Omega(G)  		@> \rho_{\bullet}>> \mathcal{D}(G) 		\\
    @V \pi_\Omega VV 			@VV \pi_\mathcal{D} V  	\\
    \Omega_U(G)			@>\rho_\bullet^U>> 	\im(\pi_\mathcal{D}) \subset \mathcal{D}_U(G) 		\\
    \end{CD}
\]
commutes.

While it is tempting to assume that both $\pi_\Omega$ and $\pi_\mathcal{D}$ are quotient maps whenever $\partial \mathrm{t}_U\subset U$ (since, by \myref{COR_rhoUHomeo}, $\rho_\bullet^U$ is a homeomorphism in this case and $[r]_U\supset [r]$ for every ray $r$ in $G$)r, we should point out that $\pi_\Omega$ (and, hence, $\pi_\mathcal{D}$) may fail to be open even when $U=\mathrm{t}(G)$: consider the graph $G = K\cup S$, where $K$ is a clique with infinitely many vertices and $S$ is a star of infinitely many rays with center $v_0\in \mathrm{V}(K)$ (see \myref{FIG_NotQuotient}). Let $U=\mathrm{t}(G)=\mathrm{V}(G\setminus K)$. Then $\partial \mathrm{t}_U\subset U=\mathrm{t}(G)$. However, let $r$ be a ray in $K$ and consider the open set $\Omega(\{v_0\},[r]) = \{[r]\}$ in $\Omega(G)$. Then $\pi_\Omega[\Omega(\{v_0\},[r])] = \{[r]_{\mathrm{t}(G)}\}$, which is not open in $\Omega_{\mathrm{t}}(G)$ because $\Omega_{\mathrm{t}}([r]_U,F)$ contains infinitely many timid-ends from $S$ for every finite $F\subset \mathrm{t}(G)$.

\end{rmk}

\begin{figure}[ht!]
    \centering
    \begin{tikzpicture}
        \def\radius{0.15}
        \def\dy{1}
        \def\dx{9}
        \def\x{0}

            \node (dots) at (\x*\dx+4*0.9,-\dy+1) {\tiny $\cdots$};
            
            \foreach \n in {0,...,4}{
                \draw[-{Latex[length=1.5mm]}] (\x*\dx,-\dy) -- ({(\x - 1 + (1.25)^(-\n))*\dx},0.8*\dy);
                \draw[-{Latex[length=1.5mm]}] (\x*\dx,-\dy) -- ({(\x + 1 - (1.25)^(-\n))*\dx},0.8*\dy);
                }

                \node (v0) at (\x*\dx-0.15,-\dy-0.1) {\tiny $v_0$};
                
                \draw[-{Latex[length=1.5mm]}] (\x*\dx,-\dy) -- (\x*\dx,-2.95*\dy);

                \foreach \i in {0,...,3}{
                    \node (K\i) at (\x*\dx,-\dy-0.5*\i) {\tiny $\bullet$};
                }
                \node (KdotsR) at (\x*\dx+0.25,-\dy-1.65) {\tiny $\vdots$};
                \node (KdotsL) at (\x*\dx-0.25,-\dy-1.65) {\tiny $\vdots$};

                \foreach \j in {2,...,3}{  
                        \draw (\x*\dx,-\dy) edge[bend left=\j*30] (\x*\dx,-\dy-0.5*\j);
                }

                \draw (\x*\dx,-\dy-0.5) edge[bend right=60] (\x*\dx,-\dy-1.5);

                \draw[dashed] (-0.75,-\dy+0.25)  rectangle (0.75,-2.8*\dy-0.5);
                \node (K) at (0.5,-2.8*\dy-0.25) {$K$};

                \draw[dashed] (-6,-\dy-0.25)  rectangle (6,\dy);
                \node (S) at (5.75,-\dy) {$S$};
                
                \draw[thick] (-6.9,-2.8*\dy-0.75)  rectangle (6.9,\dy+0.25);
                \node (G) at (6.5,\dy-0.25) {\large $G$};
            
    \end{tikzpicture}
    \caption{Figure of the graph in \myref{RMK_NotQuotient}, in which $S$ is a star of rays and $K$ is an infinite and countable clique.}
    \label{FIG_NotQuotient}
\end{figure}

    \section*{Acknowledgments}
    We thank CAPES and CNPq for the financial support. The first named author was supported by CAPES through grant number 001. The third name author was supported by CNPq through grant number 165761/2021-0. 
    
    \bibliographystyle{alpha}
    \bibliography{bibliography}

    \Addresses

    \typeout{get arXiv to do 4 passes: Label(s) may have changed. Rerun}
\end{document}